%% file: main.tex
\pdfoutput=1

\documentclass[11pt]{article}
\usepackage{enumerate}
\usepackage[OT1]{fontenc}
\usepackage[usenames]{color}
\usepackage{smile}
\usepackage[colorlinks,
            linkcolor=red,
            anchorcolor=blue,
            citecolor=blue
            ]{hyperref}
\usepackage{mathrsfs}
\usepackage{fullpage}
\usepackage{hyperref}
\usepackage[protrusion=true, expansion=true]{microtype}
\usepackage{float}
\usepackage{subfigure}
\usepackage{amsfonts,amsmath,amssymb,amsthm,url,xspace}
\usepackage{tikz}
\usepackage{verbatim}
\usetikzlibrary{arrows,shapes}
\usepackage{mathtools}
\usepackage{authblk}
\usepackage[bottom]{footmisc}
\usepackage{enumitem}
\usepackage{mathtools}
\usepackage{lscape}
\usepackage{caption}

\usepackage{algorithmicx,algpseudocode}

\ifx\counterwithout\undefined\usepackage{chngcntr}\fi
\counterwithout{equation}{section}

\allowdisplaybreaks[1]

\newcommand{\mR}{\mathbb{R}}

\newcommand{\mE}{\mathbb{E}}

\newcommand{\mL}{\mathcal{L}}

\newcommand{\mF}{\mathcal{F}}

\def\K{\mathcal{K}}
\def\P{{\mathcal P}}

\def\L{{\mathcal L}}

\def\0{{\boldsymbol 0}}

\def\bb{{\boldsymbol{b}}}
\def\N{{\mathcal N}}
\def\F{{\mathcal F}}

\def\bv{{\boldsymbol{v}}}
\def\b1{{\boldsymbol{1}}}
\def\bx{{\boldsymbol{x}}}
\def\bu{{\boldsymbol{u}}}

\def\bz{{\boldsymbol{z}}}
\def\blambda{{\boldsymbol{\lambda}}}

\def\bbeta{{\boldsymbol{\beta}}}

\def \b0{{\boldsymbol{0}}}
\def\barg{{\bar{g}}}

\def\bnabla{{\bar{\nabla}}}

\def\barblambda{{\bar{\boldsymbol{\lambda}}}}
\def\barmu{{\bar{\mu}}}

\def\hatmu{{\hat{\mu}}}

\def\barK{{\bar{K}}}
\def\bargamma{{\bar{\gamma}}}

\def\alg{TR-StoSQP }

\mathtoolsset{showonlyrefs=true}

\usepackage{xargs}
\usepackage[colorinlistoftodos,prependcaption,textsize=tiny]{todonotes}
\newcommandx{\unsure}[2][1=]{\todo[linecolor=red,backgroundcolor=red!25,bordercolor=red,#1]{#2}}
\newcommandx{\change}[2][1=]{\todo[linecolor=blue,backgroundcolor=blue!25,bordercolor=blue,#1]{#2}}
\newcommandx{\info}[2][1=]{\todo[linecolor=OliveGreen,backgroundcolor=OliveGreen!25,bordercolor=OliveGreen,#1]{#2}}
\newcommandx{\improvement}[2][1=]{\todo[linecolor=Plum,backgroundcolor=Plum!25,bordercolor=Plum,#1]{#2}}

\usepackage{alphalph}

\allowdisplaybreaks[1]

\begin{document}
\title{Fully Stochastic Trust-Region Sequential Quadratic Programming for Equality-Constrained Optimization Problems}
\author[1]{Yuchen Fang}
\author[2]{Sen Na}
\author[2,3]{Michael W. Mahoney}
\author[4]{Mladen Kolar}

\affil[1]{Committee on Computational and Applied Mathematics, The University of Chicago}
\affil[2]{ICSI and Department of Statistics, University of California, Berkeley}
\affil[3]{Lawrence Berkeley National Laboratory}
\affil[4]{Booth School of Business, The University of Chicago }

\date{}

\maketitle

\begin{abstract}
We propose a trust-region stochastic sequential quadratic programming algorithm~(TR-StoSQP) to solve nonlinear optimization problems with stochastic objectives and deterministic equality constraints. We consider a fully stochastic setting, where at each step a single sample is generated to estimate the objective gradient. The algorithm adaptively selects the trust-region radius and, compared to the existing line-search StoSQP schemes, allows us to utilize indefinite Hessian~matrices (i.e., Hessians without modification) in SQP subproblems. As a trust-region method for constrained optimization, our algorithm must address an infeasibility issue --- the linearized equality constraints and trust-region constraints may lead to infeasible SQP subproblems. In this regard, we propose an \textit{adaptive relaxation technique} to compute the trial step, consisting of a normal step and a tangential step. To control the lengths of these two steps while ensuring a scale-invariant property, we adaptively decompose the trust-region radius into two segments, based on the proportions of the rescaled feasibility and optimality residuals to the rescaled full KKT residual. The normal step~has~a~closed~form,~while the tangential step is obtained by solving a trust-region subproblem, to which a solution ensuring the Cauchy reduction is sufficient for our study. We establish a global almost sure convergence guarantee for TR-StoSQP, and illustrate its empirical performance on both a subset of problems in the CUTEst test set and constrained logistic regression problems using data from the LIBSVM collection.

\end{abstract}

\input{sec1}

\input{sec2}

\input{sec3}

\input{sec4}

\input{sec5}

\input{sec6}

\subsection*{Acknowledgments}
We would like to acknowledge the DOE, NSF, and ONR as well as the J. P. Morgan Chase Faculty Research Award for providing partial support of this work.

\bibliographystyle{my-plainnat}
\bibliography{ref}

\appendix
\numberwithin{equation}{section}
\numberwithin{theorem}{section}
\input{appendix}


\end{document}

%% file: sec1.tex
\section{Introduction}\label{sec:1}

We consider the constrained stochastic optimization problem:
\begin{equation}\label{Intro_StoProb}
\min_{\bx\in\mR^d}\;f(\bx)=\mE[F(\bx;\xi)],\quad\text{s.t.}\;\; c(\bx)=\0,
\end{equation}
\noindent where $f:\mR^d\rightarrow \mR$ is a stochastic objective with $F(\cdot;\xi)$ being one of its realizations, $c:\mR^d\rightarrow\mR^m$ are deterministic equality constraints, $\xi$ is a random variable following the distribution $\P$, and the expectation $\mE[\cdot]$ is taken over the randomness of $\xi$. Problem \eqref{Intro_StoProb} appears in various applications including constrained deep neural networks \citep{Chen2018Constraint}, constrained maximum likelihood estimation \citep{Dupacova1988Asymptotic}, optimal control \citep{Birge1997State}, PDE-constrained optimization \citep{Rees2010Optimal}, and network optimization \citep{Bertsekas1998Network}.

There are numerous methods for solving constrained optimization problems with deterministic objectives. Among them, sequential quadratic programming (SQP)~methods are one of the leading approaches and are effective for both small and large problems. When the objective is stochastic, some stochastic SQP (StoSQP) methods have been proposed recently \citep{Berahas2021Sequential, Na2022adaptive, Berahas2021Stochastic, Na2021Inequality, Curtis2021Inexact, Berahas2022Accelerating}. That body of literature considers the following two different setups for modeling the objective.

The first setup is called the random model setup \citep{Chen2017Stochastic}, where  samples with adaptive batch sizes are generated in each iteration to estimate the objective model (e.g., objective value and gradient). The algorithms under this setup often require the estimated objective model to satisfy certain adaptive accuracy conditions with a fixed probability in each iteration. Under this setup, \cite{Na2022adaptive} proposed an StoSQP algorithm~for~\eqref{Intro_StoProb}, which adopts a stochastic line search procedure with an exact augmented Lagrangian merit function to select the stepsize. Subsequently, \cite{Na2021Inequality} further enhanced the designs and arguments in \cite{Na2022adaptive} and developed an active-set StoSQP method to enable~inequality constraints; and \cite{Berahas2022Accelerating} considered a finite-sum objective and accelerated StoSQP by applying the SVRG technique \citep{Johnson2013Accelerating}, which, however, requires one to periodically compute the full objective gradient. Also, \cite{Berahas2022Adaptive} introduced a norm test condition for StoSQP to adaptively select the batch sizes.

The second setup is called the fully stochastic setup \citep{Curtis2020fully}, where a single sample is generated in each iteration to estimate the objective model. Under this setup,~a~prespecified sequence is often required as an input to assist with the step selection. For example, \cite{Berahas2021Sequential} designed an StoSQP scheme that uses a random projection \mbox{procedure} to select the stepsize. The projection procedure uses a prespecified sequence $\{\beta_k\}$, together with the estimated Lipschitz constants of the objective gradient and constraint Jacobian, to construct a projection interval in each iteration. A random quantity is then computed and projected into the interval to decide the stepsize, which ensures a sufficient reduction on the $\ell_1$ merit function. Following from \cite{Berahas2021Sequential}, some algorithmic and theoretical improvements have been reported: \cite{Berahas2021Stochastic} dealt with rank-deficient Jacobians; \cite{Curtis2021Inexact} solved Newton systems inexactly; \cite{Curtis2021Worst} analyzed the worst-case sample complexity; and \cite{Na2022Asymptotic} established the local rate and performed statistical inference for the method in \cite{Berahas2021Sequential}.

The existing StoSQP algorithms converge globally either in expectation or almost surely, and enjoy promising empirical performance under favorable settings. However, there are three limitations that motivate our study.
First, the algorithms are all line-search-based; that is, a search direction is first computed by solving an SQP subproblem, and then a stepsize is selected, either by random projection or by stochastic line search along the direction. However, it is observed that for deterministic problems, computing the search direction and selecting the stepsize jointly, as is done in trust-region methods, can lead to better performance in some cases \cite[Chapter 4]{Nocedal2006Numerical}. 
Second, to make SQP subproblems solvable, the existing schemes require the approximation of the Lagrangian Hessian to be positive definite in the null space of constraint Jacobian. Such a condition is common in the SQP literature \citep{Boggs1995Sequential, Nocedal2006Numerical}, while it is often achieved by Hessian modification, which excludes promising choices of the Hessian matrices, such as the unperturbed (stochastic) Hessian of the Lagrangian. 
Third, to show~global~convergence, the existing literature requires the random merit parameter to be not only stabilized, but also sufficiently large (or small, depending on~the~context) with an unknown threshold. To achieve the latter goal, \cite{Na2022adaptive, Na2021Inequality} imposed an adaptive condition on the feasibility error when selecting the merit parameter, while \citet{Berahas2021Stochastic,Berahas2021Sequential, Berahas2022Accelerating,Curtis2021Inexact} imposed a symmetry condition on the noise distribution.~In~contrast, deterministic SQP schemes~only require the stability of the merit parameter (see \cite{Boggs1995Sequential} and references therein).

In this paper, we consider the fully stochastic setup and design a trust-region stochastic SQP (TR-StoSQP) method to address the above limitations. As a trust-region method, TR-StoSQP computes the search direction and stepsize jointly, and, unlike line-search-based methods, it avoids Hessian modifications in formulating SQP subproblems. Thus, it can explore negative curvature directions of the Hessian. Further, our analysis only relies on the stability of the merit parameter (of the $\ell_2$ merit function), which is consistent with deterministic SQP schemes. The design of TR-StoSQP is inspired by a stochastic trust-region method for solving \textit{unconstrained} problems reported in \cite{Curtis2020fully}, which improves the authors' prior design in \cite{Curtis2019Stochastic} from using linear model to quadratic model to approximate the objective function. As in \cite{Curtis2020fully}, our method inputs a user-specified  radius-related sequence $\{\beta_k\}$ to generate the trust-region radius at each step. Beyond this similarity, our scheme differs from \cite{Curtis2020fully} in several aspects.

First, it is known that trust-region methods for constrained optimization are bothered by the \textit{infeasibility issue} --- the linearized constraints and trust-region constraints may have an empty intersection, leading to an infeasible SQP subproblem. While some literature on trust-region SQP has been proposed to address this issue \citep{Celis1984trust, Vardi1985Trust, Byrd1987Trust, Omojokun1989Trust}, we develop a novel \textit{adaptive relaxation technique} to compute the trial step, which~preserves a scale-invariant property and can be further adapted to our stochastic setup. In particular, we decompose the trial step into a normal step and a tangential step.~Then, we control the lengths of the two steps by decomposing the trust-region radius into~two segments \textit{adaptively}, based on the proportions of the rescaled estimated feasibility and optimality residuals to the rescaled full KKT residual. Compared to the existing relaxation techniques, \textit{ our relaxation technique does not require any tuning parameters}.~See Section \ref{sec:2} for details.

Second, in TR-StoSQP, we properly compute some control parameters using~known or estimable quantities. By the computation, we no longer need to tune the other~two input parameter sequences as in \cite{Curtis2020fully} (i.e., $\{\gamma_{1,k}, \gamma_{2,k}\}$ in their notation), except to~tune the input radius-related sequence $\{\beta_k\}$. Further, we use the control parameters to adjust the input sequence $\{\beta_k\}$ when computing the trust-region radius, so that $\{\beta_k\}\subseteq (0,\beta_{\max}]$ with any $\beta_{\max}>0$ is sufficient for our convergence analysis. Our design simplifies the one in \cite{Curtis2020fully}, where there are three parameter sequences to tune whose conditions are highly coupled \citep[see][Lemma 4.5]{Curtis2020fully}. In addition, as the authors stated, \cite{Curtis2020fully} rescaled the Hessian matrix based on the input $\{\gamma_{1,k}\}$, which is not ideal~(because the rescaling step modifies the curvature information of the Hessian). We have removed this step in our design.

To our knowledge, TR-StoSQP is the first trust-region SQP algorithm for solving constrained optimization problems under fully stochastic setup. With a \mbox{stabilized}~merit parameter, we establish the global convergence property of TR-StoSQP. In particular, we show that (i) when $\beta_k=\beta$, $\forall k\geq 0$, the expectation of weighted \mbox{averaged}~KKT~residuals converges to a neighborhood around zero; (ii) when $\beta_k$ decays properly such that $\sum \beta_k = \infty$ and $\sum \beta_k^2<\infty$, the KKT residuals converge to zero almost surely. These results are similar to the ones for unconstrained and constrained problems established under fully stochastic setup in \cite{Berahas2021Sequential, Berahas2021Stochastic, Curtis2021Inexact, Curtis2020fully}. However, we have weaker conditions on the objective gradient noise (e.g., we consider a growth condition) and on the sequence $\beta_k$ (e.g., we only require $\beta_k\leq \beta_{\max}$). See the discussions after Theorem \ref{cor:cons_beta} and Theorem \ref{thm:Ful_bddvar_limit} for more details. We also note that a recent paper \citep{Sun2023trust} studied a noisy trust-region method for \textit{unconstrained deterministic} optimization. In that method, the value and gradient of the objective are evaluated with bounded deterministic noise. The authors showed that the trust-region iterates visit a neighborhood of the stationarity infinitely often, with the radius proportional to the noise magnitude.~Given the significant differences between stochastic and deterministic problems, and between constrained and unconstrained problems, our~algorithm design and analysis are quite different from \cite{Sun2023trust}.  That said, when studying the stability of the merit parameter, we follow existing literature \citep[e.g.,][]{Berahas2021Sequential, Berahas2021Stochastic, Na2022adaptive} and also require the bounded gradient noise condition.~We implement TR-StoSQP on a subset of problems in the CUTEst test set and on constrained logistic regression problems using data from the LIBSVM collection.~Numerical results demonstrate the promising performance of our method.

\vskip4pt
\noindent\textbf{Notation.} We use $\|\cdot\|$ to denote the $\ell_2$ norm for vectors and the operator norm~for matrices. $I$ denotes the identity matrix and $\b0$ denotes the zero matrix (or vector). Their dimensions are clear from the context. We let $G(\bx) = \nabla^T c(\bx)\in\mR^{m\times d}$ be the Jacobian matrix of the constraints and $P(\bx)=I-G^T(\bx)[G(\bx)G^T(\bx)]^{-1}G(\bx)$ be the projection matrix to the null space of $G(\bx)$. We use $\barg(\bx) = \nabla F(\bx;\xi)$ to denote an estimate of $\nabla f(\bx)$, and use $\bar{(\cdot)}$ to denote stochastic quantities.

\vskip4pt
\noindent\textbf{Structure of the paper.}
We introduce the adaptive relaxation technique in \mbox{Section}~\ref{sec:2}. We propose the trust-region stochastic SQP (TR-StoSQP) algorithm in Section \ref{sec:3} and establish its global convergence guarantee in Section \ref{sec:4}. Numerical experiments~are~presented in Section \ref{sec:5} and conclusions are presented in Section \ref{sec:6}. Some additional~analyses are provided in Appendix \ref{appendix:A}.

%% file: sec2.tex
\section{Adaptive Relaxation for Deterministic Setup}\label{sec:2}

Lagrangian of Problem \eqref{Intro_StoProb} is $\L(\bx,\blambda)=f(\bx)+\blambda^Tc(\bx)$, where $\blambda\in\mR^m$ is the dual vector. Finding a first-order stationary point of \eqref{Intro_StoProb} is equivalent to finding a pair $(\bx^*,\blambda^*)$ such that
\begin{equation*}
\nabla \L(\bx^*,\blambda^*)=\begin{pmatrix}
\nabla_\bx \L(\bx^*,\blambda^*)\\
\nabla_{\blambda} \L(\bx^*,\blambda^*)
\end{pmatrix}=
\begin{pmatrix}
\nabla f(\bx^*)+G^T(\bx^*)\blambda^*\\
c(\bx^*)
\end{pmatrix}=\begin{pmatrix}
\b0\\
\b0
\end{pmatrix}.
\end{equation*}
\noindent We call $\|\nabla_\bx \L(\bx,\blambda)\|$ the optimality residual, $\|\nabla_\blambda \L(\bx,\blambda)\|$ (i.e., $\|c(\bx)\|$) the feasibility residual, and $\|\nabla \L(\bx,\blambda)\|$ the KKT residual. Given $\bx_k$ in the $k$-th iteration, we denote $\nabla f_k=\nabla f(\bx_k)$, $c_k = c(\bx_k)$, $G_k = G(\bx_k)$, etc.

\subsection{Preliminaries}

Given the iterate $\bx_k$ and the trust-region radius $\Delta_k$ in the $k$-th iteration, we compute an approximation $B_k$ of the Lagrangian Hessian $\nabla^2_{\bx}\L_k$,~and aim to obtain the trial step $\Delta\bx_k$ by solving a trust-region SQP subproblem 
\begin{equation}\label{def:SQPsubproblem}
\min_{\Delta\bx\in\mR^d} \ \frac{1}{2}\Delta\bx^TB_k\Delta\bx + \nabla f_k^T\Delta\bx, \; \quad\text{s.t.}\;\;
c_k+G_k\Delta\bx=\b0,\;\|\Delta\bx\|\leq\Delta_k.
\end{equation}
\noindent However, if $\{\Delta\bx\in\mR^d:c_k+G_k\Delta\bx=\b0\}\cap\{\Delta\bx\in\mR^d:\|\Delta\bx\|\leq\Delta_k\}=\emptyset$, then \eqref{def:SQPsubproblem} does not have a feasible point. This \textit{infeasibility issue} happens when the radius $\Delta_k$ is too short. To resolve this issue, one should not enlarge $\Delta_k$, which~would~make~the~trust-region constraint useless and violate the spirit of the trust-region scheme. Instead,~one should relax the linearized constraint $c_k+G_k\Delta\bx = \0$.

Before introducing our \textit{adaptive relaxation technique}, we review some classical relaxation techniques. To start, \cite{Celis1984trust} relaxed the linearized constraint by $\|c_k+G_k\Delta\bx\|\leq\theta_k$ with $\theta_k = \|c_k+G_k\Delta\bx_{k}^{CP}\|$, where $\Delta\bx_{k}^{CP}$ is the Cauchy point (i.e., the best steepest descent step) of the following problem:
\begin{equation}\label{pro:Ceils}
\min_{\Delta\bx\in\mR^d}\ \|c_k + G_k\Delta\bx\|\;
\quad\text{s.t.}\quad \|\Delta\bx\|\leq\Delta_k.
\end{equation}
\noindent However, since after the relaxation one has to minimize a quadratic function over the intersection of two ellipsoids  $\|c_k+G_k\Delta\bx\|\leq\theta_k$ and $\|\Delta\bx\|\leq \Delta_k$, the resulting SQP subproblem tends to be expensive to solve. See \cite{Yuan1990subproblem} for some insights into the difficulty, and see \cite{Yuan1991dual, Zhang1992Computing} for the methods for positive definite $B_k$. Alternatively, \cite{Vardi1985Trust} relaxed the linearized constraint by $\gamma_kc_k+G_k\Delta\bx=\b0$, with $\gamma_k\in(0,1]$ chosen to make the trust-region constraint of \eqref{def:SQPsubproblem} inactive. However, \cite{Vardi1985Trust} only showed~the~\mbox{existence}~of~an~extremely small $\gamma_k$, and it did not provide a practical way to choose it. Subsequently, \cite{Byrd1987Trust} refined the relaxation technique of \cite{Vardi1985Trust} by a step decomposition. At the $k$-th step, \cite{Byrd1987Trust} decomposed the trial step $\Delta\bx_k$ into a normal step $\bw_k\in\text{im}(G_k^T)$ and a tangential step $\bt_k\in\text{ker}(G_k)$, denoted as $\Delta\bx_k=\bw_k+\bt_k$. By the constraint $\gamma_kc_k+ G_k\Delta\bx_k  =\0$, the normal step has a closed form as (suppose $G_k$ has full row rank)
\begin{equation}\label{eq:Sto_normal_step}
\bw_k \coloneqq \gamma_k \bv_k \coloneqq -\gamma_k \cdot G_k^T[G_kG_k^T]^{-1}c_k,
\end{equation}
\noindent and the tangential step is expressed as $\bt_k=Z_k\bu_k$ for a vector $\bu_k\in\mR^{d-m}$. Here, the columns of $Z_k\in\mR^{d\times(d-m)}$ form the bases of $\text{ker}(G_k)$. \cite{Byrd1987Trust} proposed to choose $\gamma_k$ such that $\theta\Delta_k \leq \|\bw_k\|\leq \Delta_k$ for a tuning parameter $\theta\in(0,1)$, and solve $\bu_k$ from
\begin{equation}\label{eq:Byrd_tangential}
\min_{\bu\in\mR^{d-m}} \; \frac{1}{2}\bu^TZ_k^TB_kZ_k\bu+(\nabla f_k+B_k\bw_k)^TZ_k\bu
\; \quad\text{s.t.}\;\; \|\bu\|^2\leq\Delta_k^2-\|\bw_k\|^2.
\end{equation}
Furthermore, \cite{Omojokun1989Trust} combined the techniques of \cite{Celis1984trust} and \cite{Byrd1987Trust}; it solved the normal step $\bw_k$ from Problem \eqref{pro:Ceils} by replacing the constraint $\|\Delta\bx\|\leq \Delta_k$ with $\|\Delta\bx\|\leq \theta\Delta_k$~for some $\theta\in(0,1)$; and it solved the tangential step $\bt_k=Z_k\bu_k$ from Problem \eqref{eq:Byrd_tangential}. We note that the solution of \eqref{pro:Ceils} is naturally a normal step (i.e., lies in $\text{im}(G_k^T)$), because any directions in $\text{ker}(G_k)$ do not change the objective in \eqref{pro:Ceils}.

Although the methods in \cite{Byrd1987Trust, Omojokun1989Trust} allow one to employ Cauchy points for \mbox{trust-region} subproblems, they lack guidance for selecting the user-specified parameter $\theta$, which controls the lengths of the normal and tangential steps. In fact, an inappropriate~parameter $\theta$ may make either step conservative and further affect the effectiveness of~the algorithm. As we show in \eqref{eq:constraint_violation} and \eqref{eq:Ful_Cauchy_2} later, the normal step relates to the reduction of the feasibility residual, while the tangential step relates to the reduction of the optimality residual. We hope the two steps scale properly so that the model reduction achieved by $\Delta\bx_k$ is large enough. To that end, we propose an adaptive relaxation~technique, which is \textit{parameter-free} in step decomposition compared to \cite{Vardi1985Trust, Byrd1987Trust, Omojokun1989Trust}.

\subsection{Our adaptive relaxation technique}\label{sec:AdapRelax}

We introduce our parameter-free~relaxation procedure. Same as \cite{Byrd1987Trust}, we relax the linearized constraint in \eqref{def:SQPsubproblem} by~$\gamma_kc_k + G_k\Delta\bx =\0$ with $\gamma_k$ defined later, and decompose the trial step by $\Delta\bx_k = \bw_k+ \bt_k$.~The normal step $\bw_k$ is given by \eqref{eq:Sto_normal_step}, and the tangential step is of the form $\bt_k=Z_k\bu_k$.

To control the lengths of the two steps while ensuring a scale-invariant property (cf. Remark \ref{rem:2}), let us define the rescaled optimality vector $\nabla_{\bx}\L^{RS}_k\coloneqq\nabla_{\bx}\L_k/\|B_k\|$, the feasibility vector $c^{RS}_k\coloneqq c_k/\|G_k\|$, and the KKT vector $\nabla\L^{RS}_k\coloneqq(\nabla_{\bx}\L^{RS}_k,c^{RS}_k)$. { (One alternative choice of the rescaled feasibility vector can be $\bv_k = G_k^T[G_kG_k^T]^{-1}c_k$.)} Then, we \textit{adaptively} decompose the trust-region radius $\Delta_k$ into two segments, based~on the proportions of the rescaled feasibility and optimality residuals to the rescaled full KKT residual. We let 
\begin{equation}\label{eq:breve and tilde_delta_k}
\breve{\Delta}_k=\frac{\| c_k^{RS}\|}{\|\nabla\L_k^{ RS}\|}\cdot\Delta_k\quad\quad \text{ and  }\quad\quad \tilde{\Delta}_k=\frac{\|\nabla_{\bx}\L_k^{ RS}\|}{\|\nabla\L_k^{ RS}\|}\cdot\Delta_k.
\end{equation}
\noindent It is implicitly assumed that $\|B_k\|,\|G_k\|, \|\nabla\L_k\|\neq 0$, which is quite reasonable~for~SQP methods. We let $\breve{\Delta}_k$ control the length of the normal step $\bw_k$ and $\tilde{\Delta}_k$ control~the~length of the tangential step $\bt_k$. Specifically, we define $\gamma_k$ as (recall $\bv_k$ is defined in \eqref{eq:Sto_normal_step}) 
\begin{equation}\label{eq:Sto_gamma_k}
\gamma_k\coloneqq\min\{\breve{\Delta}_k/\|\bv_k\|,1 \}
\end{equation}
\noindent so that $\|\bw_k\|=\gamma_k\|\bv_k\|\leq \breve{\Delta}_k$, and we compute $\bu_k$ by solving 
\begin{equation}\label{eq:Sto_tangential_step}
\min_{\bu\in\mR^{d-m}}\ m(\bu)\coloneqq \frac{1}{2}\bu^TZ_k^TB_kZ_k\bu +  (\nabla f_k+B_k\bw_k)^TZ_k\bu \;\quad \text{ s.t. }\;\;\|\bu\|\leq\tilde{\Delta}_k.
\end{equation} 
\noindent When $\bv_k=0$ (i.e., $c_k = \0$), there is no need to choose $\gamma_k$ and we set $\Delta\bx_k=Z_k\bu_k$. Problem \eqref{eq:Sto_tangential_step} is a trust-region subproblem that appears in unconstrained~optimization. In our analysis, we only require a vector $\bu_k$ that reduces $m(\bu)$ by at least~as~much as the Cauchy point, which takes the direction of $-Z_k^T(\nabla f_k+B_k\bw_k)$ and minimizes $m(\bu)$ within the trust region \citep[Algorithm 4.2]{Nocedal2006Numerical}. Such a reduction requirement can~be achieved by various methods, including finding the exact solution or applying the dogleg or two-dimensional subspace minimization methods \citep{Nocedal2006Numerical}.

The following result provides a bound on the reduction in $m(\bu)$ that is different from the standard analysis of the Cauchy point; \citep[e.g.,][Lemma 4.3]{Nocedal2006Numerical}.

\begin{lemma}\label{lemma:Ful_cauchy}
Let $\bu_k$ be an approximate solution to \eqref{eq:Sto_tangential_step} that reduces the objective $m(\bu)$ by at least as much as the Cauchy point. For all $k\geq 0$, we have 
\begin{multline*}
 m(\bu_k)-m(\b0)= \frac{1}{2}\bu_k^TZ_k^TB_kZ_k\bu_k + (\nabla f_k+B_k\bw_k)^TZ_k\bu_k \\ \leq-\|Z_k^T(\nabla f_k+B_k\bw_k)\|\tilde{\Delta}_k+\frac{1}{2}\|B_k\|\tilde{\Delta}_k^2.
\end{multline*}
\end{lemma}

\begin{proof}
Let $\bu_k^{CP}$ denote the Cauchy point. Since $m(\bu_k) \leq m(\bu_k^{CP})$, it suffices to analyze the reduction achieved by $\bu_k^{CP}$. By the formula of $\bu_k^{CP}$ in \cite[(4.12)]{Nocedal2006Numerical}, we~know that if $\|Z_k^T(\nabla f_k+B_k\bw_k)\|^3\leq\tilde{\Delta}_k(\nabla f_k+B_k\bw_k)^TZ_kZ_k^TB_kZ_kZ_k^T(\nabla f_k+B_k\bw_k)$, then $\bu_k^{CP}=-\|Z_k^T(\nabla f_k+B_k\bw_k)\|^2/(\nabla f_k+B_k\bw_k)^TZ_kZ_k^TB_kZ_kZ_k^T(\nabla f_k+B_k\bw_k)\cdot Z_k^T(\nabla f_k+B_k\bw_k)$. In this case, using $\|Z_k\|\leq 1$, we have
\begin{multline*}
m(\bu_k^{CP})-m(\b0)=\frac{1}{2}(Z_k\bu_k^{CP})^TB_kZ_k\bu_k^{CP} +  (\nabla f_k+B_k\bw_k)^TZ_k\bu_k^{CP}\\
=-\frac{1}{2}\frac{\|Z_k^T(\nabla f_k+B_k\bw_k)\|^4}{(\nabla f_k+B_k\bw_k)^TZ_kZ_k^TB_kZ_kZ_k^T(\nabla f_k+B_k\bw_k)}\leq-\frac{1}{2}\frac{\|Z_k^T(\nabla f_k+B_k\bw_k)\|^2}{\|B_k\|}.
\end{multline*}
\noindent Otherwise, $\bu_k^{CP} = -\tilde{\Delta}_k/\|Z_k^T(\nabla f_k+B_k\bw_k)\|\cdot Z_k^T(\nabla f_k+B_k\bw_k)$. In this case, we have
\begin{align*}
m(\bu_k^{CP})&-m(\b0) =\frac{1}{2}(Z_k\bu_k^{CP})^{T}B_kZ_k\bu_k^{CP} + (\nabla f_k+B_k\bw_k)^{T}Z_k\bu_k^{CP}\\
& =\frac{(\nabla f_k+B_k\bw_k)^TZ_kZ_k^TB_kZ_kZ_k^T(\nabla f_k+B_k\bw_k)}{2\|Z_k^T(\nabla f_k+B_k\bw_k)\|^2}\tilde{\Delta}_k^2-\|Z_k^T(\nabla f_k+B_k\bw_k)\|\tilde{\Delta}_k\\
& \leq \frac{1}{2}\|B_k\|\tilde{\Delta}_k^2 -\|Z_k^T(\nabla f_k+B_k\bw_k)\|\tilde{\Delta}_k.
\end{align*}
\noindent Combining the above two cases, we have
\begin{equation*}
m(\bu_k^{CP})-m(\b0) \leq -\min\left\{-\frac{\|B_k\|\tilde{\Delta}_k^2}{2} + \|Z_k^T(\nabla f_k+B_k\bw_k)\|\tilde{\Delta}_k,\;\frac{\|Z_k^T(\nabla f_k+B_k\bw_k)\|^2}{2\|B_k\|}\right\}.
\end{equation*}
\noindent Using the fact that
\begin{multline*}
-\frac{1}{2}\|B_k\|\tilde{\Delta}_k^2 + \|Z_k^T(\nabla f_k+B_k\bw_k)\|\tilde{\Delta}_k\\ = -\frac{\|B_k\|}{2}\rbr{\tilde{\Delta}_k - \frac{\|Z_k^T(\nabla f_k+B_k\bw_k)\|}{\|B_k\|}}^2 + \frac{\|Z_k^T(\nabla f_k+B_k\bw_k)\|^2}{2\|B_k\|} \leq \frac{\|Z_k^T(\nabla f_k+B_k\bw_k)\|^2}{2\|B_k\|},
\end{multline*}
\noindent we complete the proof.
\end{proof}

It is easy to see that our relaxation technique indeed results in a trial step that~lies in the trust region. We have (noting that $\|Z_k\|\leq 1$)
\begin{equation*}
\|\Delta\bx_k\|^2=\|\bw_k\|^2+\|\bt_k\|^2=(\gamma_k\|\bv_k\|)^2+\|\bu_k\|^2\stackrel{\eqref{eq:Sto_gamma_k}, \eqref{eq:Sto_tangential_step}}{\leq}\breve{\Delta}_k^2+\tilde{\Delta}_k^2\stackrel{\eqref{eq:breve and tilde_delta_k}}{=}\Delta_k^2.
\end{equation*}
\noindent Recalling from \eqref{eq:Sto_normal_step} that $\bw_k=-\gamma_kG_k^T[G_kG_k^T]^{-1} c_k$, we know $c_k+G_k\bw_k = (1 -\gamma_k) c_k$. Thus, we have
\begin{equation}\label{eq:constraint_violation}
\|c_k+G_k\Delta\bx_k\|-\|c_k\|=\|c_k+G_k\bw_k\|-\|c_k\|=-\gamma_k\|c_k\|\leq 0,
\end{equation}
\noindent where the strict inequality holds as long as $c_k\neq \b0$. This inequality suggests that~the normal step $\bw_k$ helps to reduce the feasibility residual. Furthermore, when we define the least-squares Lagrangian multiplier as $\boldsymbol{\lambda}_k=-G_k^T[G_kG_k^T]\nabla f_k$, we have $P_k\nabla f_k=\nabla_{\bx}\L_k$. Noting that $Z_kZ_k^T=P_k$, $P_k^2=P_k$ and $Z_k^TZ_k=I$, we obtain 
\begin{align*}
\|Z_k^T(\nabla f_k+B_k\bw_k)\|^2 & =  (\nabla f_k+B_k\bw_k)^T Z_k Z_k^T(\nabla f_k+B_k\bw_k)\\
& =(\nabla f_k+B_k\bw_k)^T P_k^2(\nabla f_k+B_k\bw_k)=\|\nabla_{\bx}\L_k+P_kB_k\bw_k\|^2.
\end{align*}
\hskip-3pt Thus, the conclusion of Lemma~\ref{lemma:Ful_cauchy} can be rewritten as
\begin{equation}\label{eq:Ful_Cauchy_2}
m(\bu_k)-m(\b0)\leq -\|\nabla_{\bx}\L_k + P_kB_k\bw_k\|\tilde{\Delta}_k+\frac{1}{2}\|B_k\|\tilde{\Delta}_k^2,
\end{equation}
\noindent indicating that the tangential step relates to the reduction of the optimality residual.

To end this section, we would like to link our relaxation~technique with those~in \cite{Byrd1987Trust, Omojokun1989Trust} in Remark \ref{rem:2}. 

\begin{remark}\label{rem:2}
In our method, we define \textit{rescaled} residuals $\|\nabla_{\bx}\L^{RS}_k\|$, $\|c^{RS}_k\|$,~$\|\nabla\L^{RS}_k\|$, and adaptively decompose the radius based on the proportions of these rescaled residuals (cf. \eqref{eq:breve and tilde_delta_k}). We have two motivations: (i) the relation of the normal and tangential steps to the feasibility and optimality residuals; (ii) a scale-invariant property. We~explain as follows.

Seeing from \eqref{eq:constraint_violation} and \eqref{eq:Ful_Cauchy_2}, the normal step relates to the reduction of the feasibility residual, while the tangential step relates to the reduction of the \mbox{optimality}~residual. When the proportion of the feasibility residual is larger than that of the optimality residual, decreasing the feasibility residual is more important. As a result, we assign~a larger trust-region radius to the normal step to achieve a larger reduction~in~the feasibility residual. Otherwise, we assign a larger radius to the tangential step to achieve a larger reduction in the optimality residual. In comparison, \cite{Byrd1987Trust, Omojokun1989Trust}~rely on a fixed~proportion constant $\theta\in(0,1)$, making their approach less adaptive than ours.

 On the other hand, we note that \cite{Byrd1987Trust, Omojokun1989Trust} enjoy a nice scale-invariant property:~given the radius $\Delta_k$, the trial step $\Delta\bx_k$ is invariant when the constraints $c$  and/or the objective $f$ are rescaled by a (positive) scalar. Note that if $f$ (or $c$) is rescaled by~a~positive scalar, the Lagrangian Hessian (or the constraints Jacobian) will be rescaled by the same scalar. To preserve the invariance property, we decompose~$\Delta_k$~using the rescaled residuals, as opposed to the original residuals $\|\nabla_{\bx}\mL_k\|$ and $\|c_k\|$; the latter can never be scale-invariant.
\end{remark}
In the next section, we move to the fully stochastic setup and utilize the proposed relaxation scheme to design an StoSQP algorithm for \eqref{Intro_StoProb}. We will also discuss how to use the relaxation in \cite{Byrd1987Trust} to design a StoSQP method.

%% file: sec3.tex
\section{A Trust-Region Stochastic SQP Algorithm}\label{sec:3}

From now on, we replace the deterministic gradient $\nabla f(\bx)$ by its stochastic estimate $\barg(\bx) = \nabla F(\bx; \xi)$. Similar to Section \ref{sec:2}, we denote $\barg_k=\barg(\bx_k)$ and define the \textit{estimated} KKT residual as $\|\bnabla\mL_k\|=\|(\bnabla_\bx\mL_k, c_k)\|$ with $\bnabla_\bx\mL_k = \barg_k + G_k^T\blambda_k$.

We summarize the proposed TR-StoSQP algorithm in Algorithm \ref{Alg:Non-adap}, and introduce the algorithm details as follows. In the $k$-th iteration, we are given the iterate~$\bx_k$, two fixed scalars $\zeta>0$ and $\delta\geq 0$, and the parameters $(\beta_k,L_{\nabla f,k}, L_{G,k}, \barmu_{k-1})$. Here, $\beta_k\in(0,\beta_{\max}]$ with upper bound $\beta_{\max}>0$ is the input radius-related parameter;~$L_{\nabla f,k}$ and $L_{G,k}$ are the (estimated) Lipschitz constants of $\nabla f(\bx)$ and $G(\bx)$ (in~practice, they can be estimated by standard procedures in \cite{Curtis2018Exploiting, Berahas2021Sequential}); and $\barmu_{k-1}$ is the merit parameter of the $\ell_2$ merit function obtained after the $(k-1)$-th iteration. With these parameters, we proceed with the following three steps. 

\vskip4pt

\noindent\textbf{Step 1: Compute control parameters.}
We compute a matrix $B_k$ to approximate the Hessian of the Lagrangian $\nabla^2_{\bx}\L_k$, and require it to be deterministic conditioning on $\bx_k$. With $\bv_k$ defined in \eqref{eq:Sto_normal_step}, we then compute several control parameters:
\begin{equation}\label{def:eta2k}
\begin{aligned}
\eta_{1,k} &  = \zeta\cdot \|\bv_k\|/\|c_k\|, \qquad \qquad\;\tau_k=L_{\nabla f,k}+L_{G,k}\barmu_{k-1}+\|B_k\|, \\
\alpha_k &  = \frac{\beta_k}{4(\eta_{1,k}\tau_k+\zeta)\beta_{\max}},\hskip0.5cm \eta_{2,k} = \eta_{1,k}-\frac{1}{2}\zeta\eta_{1,k}\alpha_k.
\end{aligned}
\end{equation}
\noindent We should emphasize that, compared to the existing line-search-based StoSQP methods \citep{Berahas2021Sequential, Berahas2021Stochastic, Berahas2022Accelerating, Na2022adaptive, Na2021Inequality, Na2022Asymptotic}, we do not require $B_k$ to be positive definite in the null space~$\text{ker}(G_k)$. This benefit adheres to the trust-region methods, more precisely, the existence of the trust-region constraint. Due to this benefit, we can construct different $B_k$ to formulate the StoSQP subproblems. In our experiments in Section \ref{sec:5}, we will construct $B_k$~by~the identity matrix, the symmetric rank-one (SR1) update, the estimated Hessian without modification, and the average of the estimated Hessians.

The control parameters in \eqref{def:eta2k} play a critical role in adjusting the input $\{\beta_k\}$~and generating the trust-region radius. Compared to \cite{Curtis2020fully}, $\{\eta_{1,k},\eta_{2,k}\}$ (i.e., $\{\gamma_{1,k}, \gamma_{2,k}\}$ in their notation) are no longer inputs and $B_k$ is not rescaled by the parameters.

\vskip4pt

\noindent\textbf{Step 2: Compute the trust-region radius.} We sample a realization $\xi_g^k$ and~compute an estimate $\barg_k=\nabla F(\bx_k;\xi_g^k)$ of $\nabla f_k$. We then compute the least-squares Lagrangian multiplier as  $\barblambda_k=-[G_kG_k^T]^{-1}G_k\barg_k$ and the KKT vector $\bar{\nabla}\L_k$. Furthermore, we define the trust-region radius as 
\begin{equation}\label{Ful_GenerateRadius}
\Delta_k=\left\{
\begin{aligned}
&\eta_{1,k}\alpha_k\|\bar{\nabla}\L_k\|  \quad \text{if }\|\bar{\nabla}\L_k\|\in(0,1/\eta_{1,k}), \\
&\alpha_k\qquad\qquad\quad\;\; \text{if }\|\bar{\nabla}\L_k\|\in[1/\eta_{1,k},1/\eta_{2,k}], \\
&\eta_{2,k}\alpha_k\|\bar{\nabla}\L_k\|\quad\text{if }\|\bar{\nabla}\L_k\|\in(1/\eta_{2,k},\infty).
\end{aligned}
\right.
\end{equation}
\noindent We provide the following remark to compare \eqref{Ful_GenerateRadius} with the line search scheme in \cite{Berahas2021Sequential}.

\begin{remark}\label{rem:1}
It is interesting to see that the scheme \eqref{Ful_GenerateRadius} enjoys the same flavor as the random-projection-based line search procedure in \cite{Berahas2021Sequential}. In particular, \cite{Berahas2021Sequential} updates $\bx_k$ by $\alpha_k\tilde{\Delta}\bx_k$ each step, where $\tilde{\Delta}\bx_k$ is solved from Problem \eqref{def:SQPsubproblem} (without trust-region constraint) and the stepsize $\alpha_k$ is selected by projecting a random~\mbox{quantity}~into an~interval like $[\beta_k, \beta_k+\beta_k^2]$ (see \eqref{def:proj} below). By the facts that $\|\tilde{\Delta}\bx_k\|=\mathcal{O}(\|\bar{\nabla}\L_k\|)$~(i.e., $\tilde{\Delta}\bx_k$ and $\bnabla\L_k$ have the same order of magnitude)~and~$\alpha_k = \mathcal{O}(\beta_k)$, we know $\|\bx_{k+1}-\bx_k\| =\|\alpha_k\tilde{\Delta}\bx_k\|= \mathcal{O}(\beta_k\|\bar{\nabla}\L_k\|)$. This order is preserved by our trust-region scheme since, seeing from \eqref{def:eta2k} and \eqref{Ful_GenerateRadius}, we have $\|\bx_{k+1}-\bx_k\| = \|\Delta\bx_k\| = \mathcal{O}(\beta_k\|\bar{\nabla}\L_k\|)$.~Furthermore, the projection in \cite{Berahas2021Sequential} brings some sort of adaptivity to the scheme as the stepsize $\alpha_k$ has a variability of $\mathcal{O}(\beta_k^2)$. This merit is also preserved by \eqref{Ful_GenerateRadius}, noting~that $(\eta_{1,k}-\eta_{2,k})\alpha_k = \mathcal{O}(\beta_k^2)$.

We emphasize that \eqref{Ful_GenerateRadius} offers adaptivity to selecting the radius $\Delta_k$ based on $\alpha_k(= \mathcal{O}(\beta_k))$. When $\|\bar{\nabla}\L_k\|$ is large, the iterate $\bx_k$ is likely to be far from the~KKT~point. Thus, we set $\Delta_k>\alpha_k$ to be more aggressive than $\alpha_k$. Otherwise, when $\|\bar{\nabla}\L_k\|$~is~small, the iterate $\bx_k$ is likely to be near the KKT point. Thus, we set $\Delta_k<\alpha_k$ to be more conservative than $\alpha_k$.

\end{remark}

\noindent\textbf{Step 3: Compute the trial step and update the merit parameter.} With $\Delta_k$ from Step 2, we adapt the relaxation technique in Section \ref{sec:AdapRelax} to compute the trial step $\Delta\bx_k = \bw_k+\bt_k$. In particular, we apply \eqref{eq:breve and tilde_delta_k} to decompose $\Delta_k$, with deterministic residuals replaced by their stochastic estimates. Then, we apply \eqref{eq:Sto_gamma_k} to compute the stochastic counterpart of $\gamma_k$, denoted as $\bargamma_k^{\text{trial}}$. Then, we set $\bargamma_k$ as
\begin{equation}\label{def:proj}
\bargamma_k\leftarrow \text{Proj}\left(\bargamma_k^{\text{trial}}\big|\left[0.5\zeta\phi_k\alpha_k,0.5\zeta\phi_k\alpha_k+\delta\alpha_k^2\right]\right),
\end{equation}
\noindent where $\phi_k=\min\{\|B_k\|/\|G_k\|,1\}$ and $\text{Proj}(a|[b,c])$ is the projection function. It equals $a$ if $a\in[b,c]$, $b$ if $a<b$, and $c$ if $a>c$. The normal step is $\bw_k=\bargamma_k\bv_k$, and the tangential step $\bt_k = Z_k\bu_k$ is solved from \eqref{eq:Sto_tangential_step},  achieving an reduction at~least~as~much~as Cauchy reduction. Finally, we update the iterate as $\bx_{k+1}=\bx_k+\Delta\bx_k$, and update~the merit parameter $\barmu_{k-1}$ of the $\ell_2$ merit function, defined as
\begin{equation}\label{equ:merit}
\L_{\barmu}(\bx) = f(\bx) + \barmu\|c(\bx)\|.
\end{equation}
\noindent Specifically, we let $\barmu_k = \barmu_{k-1}$ and compute the predicted reduction of $\mL_{\barmu_k}^k$ as
\begin{equation}\label{eq:Ful_Pred_k}
\text{Pred}_k=\barg_k^T\Delta\bx_k+\frac{1}{2}\Delta\bx_k^TB_k\Delta\bx_k+\barmu_k(\|c_k+G_k\Delta\bx_k\|-\|c_k\|).
\end{equation}
\noindent The parameter $\barmu_k$ is then iteratively updated as $\barmu_k \leftarrow \rho \barmu_k$ with some $\rho>1$ until
\begin{equation}\label{eq:upper_bound_predk}
\text{Pred}_k\leq  -\|\bar{\nabla}\L_k\|\Delta_k+\frac{1}{2}\|B_k\|\Delta_k^2.
\end{equation}
\noindent We now explain some components of Step 3 in the following remarks.

\begin{remark}
The update rule for the merit parameter in \eqref{eq:upper_bound_predk} is well-posed and~terminates in finite number of steps.~By \eqref{eq:constraint_violation}, $\barmu_k(\|c_k+G_k\Delta\bx_k\|-\|c_k\|)=-\bargamma_k\barmu_k\|c_k\|$.~Thus, when $\|c_k\|\neq 0$, $\text{Pred}_k$ decreases as $\barmu_k$ increases and \eqref{eq:upper_bound_predk} is satisfied for a sufficiently large $\barmu_k$. When $\|c_k\|= 0$, both $\bw_k$ and $\breve{\Delta}_k$ vanish, and $\text{Pred}_k = m(\bu_k)-m(\b0)$. Then, \eqref{eq:upper_bound_predk} is satisfied solely by the tangential step, without selecting the merit parameter,~as can be seen from \eqref{eq:Ful_Cauchy_2}.
The choice of the right-hand-side \mbox{threshold}~of~\eqref{eq:upper_bound_predk}~\mbox{ensures}~that the trial step achieves a sufficient reduction on the merit function \eqref{equ:merit}. In particular, it is known for SQP methods that the predicted reduction of the merit \mbox{function}~is~characterized by the directional derivative of the merit function along the trial step,~which is proportional to $-\|\bnabla\mL_k\|^2$ when the merit parameter $\barmu_k$ is selected~properly \citep[see][]{Berahas2021Sequential, Na2022adaptive}. This motivates the first term of the threshold.~\mbox{Further},~to~control the~quadratic term $\Delta\bx_k^TB_k\Delta\bx_k/2$ in \eqref{eq:Ful_Pred_k}, we offset the threshold by the second term $\|B_k\|\Delta_k^2/2$, which stems from the positive term in Cauchy reduction (see Lemma \ref{lemma:Ful_cauchy}). Overall, as shown in Lemma \ref{lemma:Ful_cond2}, the right-hand-side of \eqref{eq:upper_bound_predk} is always negative, meaning that the trial steps leads to a sufficient reduction. 

The iterative update $\barmu_k \leftarrow \rho \barmu_k$ is not essential since the threshold of $\barmu_k$ can be~obtained by directly solving \eqref{eq:upper_bound_predk}. Then, $\barmu_k$ can be updated by taking the maximum~between $\rho\barmu_k$ and the threshold. The maximum operation ensures that~$\barmu_k$~is~\mbox{increased}~by at least a fixed amount, $\rho\barmu_{-1}$,~\mbox{whenever}~it~is~\mbox{updated}.~This is important for the stability result of $\barmu_k$ (see Lemma \ref{lemma:predk}).

\end{remark}

\begin{algorithm}[t]
\caption{A Trust Region Stochastic SQP (TR-StoSQP) Algorithm}\label{Alg:Non-adap}
\begin{algorithmic}[1]
\State \textbf{Input:} Initial iterate $\bx_0$, radius-related sequence $\{\beta_k\}\subset(0,\beta_{\max}]$, parameters $\rho>1,\barmu_{-1},\zeta>0$, { $\delta\geq 0$}, (estimated) Lipschitz constants $\{L_{\nabla f,k}\},\{L_{G,k}\}$.
\For {$k=0,1,\cdots,$}
\State Compute an approximation $B_k$ and control parameters $\eta_{1,k},\tau_k,\alpha_k,\eta_{2,k}$ as \eqref{def:eta2k};
\State Sample $\xi_g^k$ and compute $\barg_k$, $\barblambda_k$, $\bar{\nabla}\L_k$, and the trust-region radius $\Delta_k$ as \eqref{Ful_GenerateRadius};
\State Decompose $\Delta_k$ as \eqref{eq:breve and tilde_delta_k} and compute $\bargamma_k^{\text{trial}}$ as \eqref{eq:Sto_gamma_k} and $\bargamma_k$ as \eqref{def:proj};
\State Compute $\Delta\bx_k = \bw_k +\bt_k$, where $\bw_k = \bargamma_k\bv_k$ and $\bt_k=Z_k\bu_k$ is from \eqref{eq:Sto_tangential_step};
\State Update $\bx_{k+1}=\bx_k+\Delta\bx_k$, set $\barmu_k=\barmu_{k-1}$, and compute $\text{Pred}_k$ as \eqref{eq:Ful_Pred_k};
\While{\eqref{eq:upper_bound_predk} does not hold}
\State Set $\barmu_k=\rho\barmu_k$;
\EndWhile
\EndFor
\end{algorithmic}
\end{algorithm}
\setlength{\textfloatsep}{0.3cm}

\begin{remark}
We utilize a projection step \eqref{def:proj} in the selection of $\bargamma_k$. The interval with a length of $\delta\alpha_k^2$ provides some sort of flexibility in the selection, similar~to~\cite{Berahas2021Stochastic,Berahas2021Sequential} and references therein.~The~motivation behind the projection is to~\mbox{regulate}~$\bargamma_k$~\mbox{using}~control parameters~\mbox{computed}~in~\eqref{def:eta2k}. To gain insight into the interval~\mbox{boundary},~we~consider a small $\alpha_k$. Combining \eqref{eq:breve and tilde_delta_k}, \eqref{eq:Sto_gamma_k}, and \eqref{Ful_GenerateRadius}, we obtain that $\bargamma_k^{\text{trial}} = \breve{\Delta}_k/\|\bv_k\| =\mathcal{O}(\Delta_k/\|\bnabla\mL_k^{RS}\|) =\mathcal{O}(\alpha_k)$.
As a result, the boundary should scale proportionally with $\alpha_k$. However, $\mathcal{O}(\cdot)$ hides~the~ratios between unscaled and scaled residuals, such as $\|\bnabla\mL_k\|/\|\bnabla\mL_k^{RS}\|$. The control~parameters are utilized to offer a deterministic lower bound for these ratios. In the end, we can show that (see \eqref{equ:4.13})
\begin{equation*}
\zeta\phi_k\alpha_k/2\leq \min\{\breve{\Delta}_k/\|\bv_k\|,1\} \eqqcolon \bargamma_k^{\text{trial}},
\end{equation*}
\noindent which implies $\bargamma_k\leq \bargamma_k^{\text{trial}}$ and, consequently, the normal step $\|\bw_k\| = \bargamma_k\|\bv_k\|\leq \breve{\Delta}_k$.
\end{remark}

\begin{remark}\label{rem:3}
In addition to our adaptive relaxation technique, we consider two~alternative relaxation approaches for designing StoSQP methods.~These approaches~only affect the computation of $\Delta\bx_k$, while the remaining parts of the algorithm remain the same.~Thus, these approaches enjoy the same global \mbox{convergence}~analysis. The proof of the stability result of the merit parameter $\barmu_k$ may differ slightly. In this regard,~the detailed analysis is provided in Appendix \ref{appendix:A} for the sake of completeness. We empirically investigate the performance of the following methods in Section \ref{sec:5}.

\vskip4pt

(i) We compute the same normal step $\bw_k$, but instead of using \eqref{eq:Sto_tangential_step} to \mbox{compute}~the tangential step $\bt_k$, we follow the approach \citep{Byrd1987Trust,Omojokun1989Trust} and use \eqref{eq:Byrd_tangential}. In other words, we~do not decompose $\Delta_k$ as in \eqref{eq:breve and tilde_delta_k}, but define $\tilde{\Delta}_k\coloneqq \sqrt{\Delta_k^2-\|\bw_k\|^2}$.

(ii) We follow the approach in \cite{Byrd1987Trust}. In particular, we decompose $\Delta_k$ as $\breve{\Delta}_k \coloneqq\theta\Delta_k$ and $\tilde{\Delta}_k\coloneqq \sqrt{\Delta_k^2-\|\bw_k\|^2}$ for a prespecified constant $\theta\in(0,1]$; and apply~Algorithm \ref{Alg:Non-adap} to derive the normal and tangential steps with $\phi_k$ in \eqref{def:proj} replaced by $\theta$.
\end{remark}

We end this section by introducing the randomness in TR-StoSQP. We let $\F_0\subseteq\F_1\subseteq\F_2\cdots$ be a filtration of $\sigma$-algebras with $\F_{k-1}$ generated by $\{\xi_g^j\}_{j=0}^{k-1}$; thus, $\F_{k-1}$ contains all the randomness before the $k$-th iteration. Let $\F_{-1}=\sigma(\bx_0)$ be the trivial $\sigma$-algebra for consistency. It is easy to see that for all $k\geq 0$, we have
\begin{equation*}
\sigma(\bx_k,\eta_{1,k},\tau_k,\alpha_k,\eta_{2,k})\subseteq \F_{k-1}\quad
\text{ and }\quad \sigma(\Delta\bx_k, \barblambda_k, \barmu_k)\subseteq \F_k.
\end{equation*}
\noindent In the next section, we conduct the global analysis of the proposed algorithm.

%% file: sec4.tex
\section{Convergence Analysis}\label{sec:4}

We study the convergence of Algorithm \ref{Alg:Non-adap} by measuring the decrease of the $\ell_2$ merit function at each step, that is
\begin{equation*}
\L_{\barmu_k}^{k+1}-\L_{\barmu_k}^k=f_{k+1}-f_k+\barmu_k(\|c_{k+1}\|-\|c_k\|).
\end{equation*}
\noindent We use $\barmu_k$ to denote the merit parameter obtained after the While loop in Line 10 of Algorithm \ref{Alg:Non-adap}, so that $\barmu_k$ satisfies \eqref{eq:upper_bound_predk}. Following the analysis of \cite{Berahas2021Sequential, Berahas2021Stochastic, Berahas2022Accelerating, Curtis2021Inexact}, we will first assume $\barmu_k$ stabilizes (but not necessarily at a large enough value) after~a~few~iterations, and then we will validate the stability of $\barmu_k$ in Section \ref{sec:penalty}. 

We now state the assumptions for the analysis.

\begin{assumption}\label{ass:1-1}

Let $\Omega\subseteq\mR^d$ be an open convex set containing the iterates $\{\bx_k\}$. The function $f(\bx)$ is continuously differentiable and is bounded below by $f_{\inf}$ over $\Omega$. The gradient $\nabla f(\bx)$ is Lipschitz continuous over $\Omega$ with constant $L_{\nabla f}>0$, so that the (estimated) Lipschitz constant $L_{\nabla f, k}$ at $\bx_k$ satisfies $L_{\nabla f, k}\leq L_{\nabla f}$, $\forall k\geq0$. Similarly, the constraint $c(\bx)$ is continuously differentiable over $\Omega$; its Jacobian $G(\bx)$ is Lipschitz continuous over $\Omega$ with constant $L_G>0$; and $L_{G,k}\leq L_G$, $\forall k\geq 0$. We also assume there exist positive constants $\kappa_B,\kappa_c,\kappa_{\nabla f},\kappa_{1,G},\kappa_{2,G}>0$ such that 
\begin{equation*}
\|B_k\|\leq \kappa_B,\;\; \|c_k\|\leq \kappa_{c}, \;\; \|\nabla f_k\|\leq\kappa_{\nabla f},\;\;\; \kappa_{1,G}\cdot I\preceq G_kG_k^T\preceq\kappa_{2,G}\cdot I, \;\;\; \forall k\geq 0.
\end{equation*}
\end{assumption}

Assumption \ref{ass:1-1} is standard in the literature on both deterministic and stochastic SQP methods;  \citep[see, e.g.,][]{Byrd1987Trust,ElAlem1991Global,Powell1990trust, Berahas2021Sequential, Berahas2021Stochastic, Berahas2022Accelerating, Curtis2021Inexact}. In fact, when one uses a While loop~to~adaptively increase $L_{\nabla f,k}$ and $L_{G,k}$ to enforce the Lipschitz conditions (as did in \cite{Berahas2021Sequential, Curtis2018Exploiting}),~one has $L_{\nabla f,k}\leq L_{\nabla f}' \coloneqq \rho L_{\nabla f}$ for a factor $\rho>1$ \cite[same for $L_{G,k}$; see][Lemma 8]{Berahas2021Sequential}. We unify the Lipschitz constant and upper bound of $L_{\nabla f,k}$ as $L_{\nabla f}$ just for simplicity. In addition, the condition $\kappa_{1,G}\cdot I\preceq G_kG_k^T\preceq\kappa_{2,G}\cdot I$ implies $G_k$ has full row rank; thus, the least-squares dual iterate $\barblambda_k=-[G_kG_k^T]^{-1}G_k\barg_k$ is well defined.

Next, we assume the stability of $\barmu_k$. Compared to existing StoSQP literature~\citep{Berahas2021Sequential, Berahas2021Stochastic, Berahas2022Accelerating, Curtis2021Inexact}, we do not require the stabilized value to be large enough. We will revisit this assumption in Section \ref{sec:penalty}.

\begin{assumption}\label{ass:stabilized_penalty}
There exist an (possibly random) iteration threshold $\barK<\infty$~and a deterministic constant $\hat{\mu}>0$, such that for all $k>\barK$, $\barmu_k=\barmu_{\barK}\leq \hatmu$. 
\end{assumption}

Since $\barmu_k$ is non-decreasing in TR-StoSQP, we have $\barmu_k\leq\hatmu$, $\forall k\geq 0$. The global analysis only needs to study the convergence behavior of the algorithm after $k\geq\barK+1$ iterations. Next, we impose a condition on the gradient estimate.

\begin{assumption}\label{ass:Ful_unbias}
There exist constants $M_g\geq 1,M_{g,1}\geq 0$ such that the stochastic gradient estimate $\bar{g}_k$ satisfies $\mE_k[\bar{g}_k]=\nabla f_k$ and $\mE_k[\|\barg_k-\nabla f_k\|^2]\leq M_g+M_{g,1}(f_k-f_{\inf})$, $\forall k\geq0$, where $\mE_k[\cdot]$ denotes $\mE[\cdot\mid\mF_{k-1}]$.
\end{assumption}

We assume that the variance of the gradient estimate satisfies a growth condition. This condition is weaker than the usual bounded variance condition assumed in the StoSQP literature \citep{Curtis2020fully, Berahas2021Stochastic,Berahas2021Sequential,Na2021Inequality,Na2022adaptive}, which corresponds to $M_{g,1}=0$. The growth condition is more realistic and was recently investigated for stochastic first-order methods on unconstrained problems \citep{Stich2019Unified,Bottou2018Optimization,Vaswani2019Fast, Chen2020Convergence}, while is less explored for StoSQP methods.

\input{sec4_fundamental_lemma}

\input{sec4_global_convergence}

\input{sec4_merit}

%% file: sec4_fundamental_lemma.tex
\subsection{Fundamental lemmas}\label{subsec:fund_lemma}

The following result establishes the reduction of~the $\ell_2$ merit function achieved by the trial step.

\begin{lemma}\label{lemma:ared_k}
Suppose Assumptions \ref{ass:1-1} and \ref{ass:stabilized_penalty} hold. For all $k\geq\barK+1$, we have
\begin{equation}\label{eq:ared_k}
\L_{\barmu_{\barK}}^{k+1}-\L_{\barmu_{\barK}}^k \leq  -\|\bar{\nabla}\L_k\|\Delta_k+\frac{1}{2}\|B_k\| \Delta_k^2 
+\bargamma_k(\nabla f_k-\barg_k)^T\bv_k +\|P_k(\nabla f_k-\barg_k)\|\Delta_k+\frac{1}{2}\tau_k \Delta_k^2.
\end{equation}
\end{lemma}

\begin{proof}
By the definitions of $\L_{\barmu_{\barK}}(\bx)$ and $\text{Pred}_k$ in \eqref{equ:merit} and \eqref{eq:Ful_Pred_k}, we have
\begin{equation*}
\L_{\barmu_{\barK}}^{k+1}-\L_{\barmu_{\barK}}^k-\text{Pred}_k = f_{k+1}-f_k-\barg_k^T\Delta\bx_k-\frac{1}{2}\Delta\bx_k^TB_k\Delta\bx_k+\barmu_{\barK}(\|c_{k+1}\|-\|c_k+G_k\Delta\bx_k\|).
\end{equation*}
\noindent By the Lipschitz continuity of $\nabla f(\bx)$ and $G(\bx)$, we further have
\begin{align*}
\L_{\barmu_{\barK}}^{k+1} & -\L_{\barmu_{\barK}}^k-\text{Pred}_k\leq (\nabla f_k-\barg_k)^T\Delta\bx_k+ \frac{1}{2}(L_{\nabla f,k}+\|B_k\|+L_{G,k}\barmu_{\barK})\|\Delta\bx_k\|^2\\
& \stackrel{\mathclap{\eqref{def:eta2k}}}{=}\; (\nabla f_k-\barg_k)^T\Delta\bx_k+\frac{1}{2}\tau_k\|\Delta\bx_k\|^2\\
& = \bargamma_k(\nabla f_k-\barg_k)^T\bv_k+(\nabla f_k-\barg_k)^TZ_k\bu_k+
\frac{1}{2}\tau_k\|\Delta\bx_k\|^2\;\;\;(\Delta\bx_k=\bargamma_k\bv_k+Z_k\bu_k)\\
&\leq  \bargamma_k(\nabla f_k-\barg_k)^T\bv_k+\|P_k(\nabla f_k-\barg_k)\|\|\bu_k\|+
\frac{1}{2}\tau_k\|\Delta\bx_k\|^2,
\end{align*}
\noindent where the last inequality uses $Z_kZ_k^T = P_k$. Combining the above result with the~reduction condition in \eqref{eq:upper_bound_predk}, and noting that $\|\bu_k\|\leq \|\Delta\bx_k\|\leq \Delta_k$, we complete the~proof.
\end{proof}

Now, we further analyze the right-hand-side of \eqref{eq:ared_k}. By taking the expectation conditional on $\bx_k$, we can show that the term $\bargamma_k(\nabla f_k-\barg_k)^T\bv_k$ is upper bounded by~a quantity proportional to the expected error of the gradient estimate.

\begin{lemma}\label{lemma:exp<=0}
Suppose Assumptions \ref{ass:1-1} and \ref{ass:Ful_unbias} hold. For all~$k\geq 0$, we have
\begin{equation*}
\mE_k[\bargamma_k(\nabla f_k-\barg_k)^T\bv_k]\leq \frac{\delta\kappa_c}{\sqrt{\kappa_{1,G}}} \alpha_k^2\cdot \mE_k[\|\nabla f_k-\barg_k\|].
\end{equation*}
\end{lemma}

\begin{proof}
When $\bv_k=\b0$, the result holds trivially. We consider $\bv_k\neq\b0$. By the~design of the projection in \eqref{def:proj}, we know
\begin{equation}\label{pequ:1}
\gamma_{k,\min}\coloneqq \frac{1}{2}\zeta\phi_k\alpha_k\leq \bargamma_k\leq \frac{1}{2}\zeta\phi_k\alpha_k+\delta\alpha_k^2 \eqqcolon \gamma_{k,\max}.
\end{equation}
\noindent Note that $\sigma(\gamma_{k,\min},\gamma_{k,\max})\subseteq\mF_{k-1}$. Let $E_k$ be the event that $(\nabla f_k-\barg_k)^T\bv_k\geq0$, $E_k^c$ be its complement, and $\mathbb{P}_k[\cdot]$ denote the probability conditional on $\mF_{k-1}$. By the law of total expectation, one finds
{\allowdisplaybreaks\begin{align*}
\mE_k& [\bargamma_k(\nabla f_k-\barg_k)^T\bv_k]\\
& = \mE_k[\bargamma_k(\nabla f_k-\barg_k)^T\bv_k \mid E_k]\mathbb{P}_k[E_k]+\mE_k[\bargamma_k(\nabla f_k-\barg_k)^T\bv_k \mid E_k^c]\mathbb{P}_k[E_k^c] \\
& \stackrel{\mathclap{\eqref{pequ:1}}} {\leq} \gamma_{k,\max}\mE_k[(\nabla f_k-\barg_k)^T\bv_k \mid E_k]\mathbb{P}_k[E_k] + \gamma_{k,\min}\mE_k[(\nabla f_k-\barg_k)^T\bv_k \mid E_k^c]\mathbb{P}_k[E_k^c]\\
& = (\gamma_{k,\max}-\gamma_{k,\min})\mE_k[(\nabla f_k-\barg_k)^T\bv_k \mid E_k]\mathbb{P}_k[E_k]\quad (\text{by Assumption } \ref{ass:Ful_unbias})\\
& \leq (\gamma_{k,\max}-\gamma_{k,\min})\mE_k[\|\nabla f_k-\barg_k\|\|\bv_k\| \mid E_k]\mathbb{P}_k[E_k]\\
& \leq (\gamma_{k,\max}-\gamma_{k,\min})\|\bv_k\|\mE_k[\|\nabla f_k-\barg_k\|]\\
& \stackrel{\mathclap{\eqref{pequ:1}}}{=} \;
\delta\alpha_k^2\|\bv_k\|\mE_k[\|\nabla f_k-\barg_k\|]\stackrel{\mathclap{\eqref{def:eta2k}}}{\leq}\;\frac{\delta\kappa_c}{\sqrt{\kappa_{1,G}}}\alpha_k^2\mE_k[\|\nabla f_k-\barg_k\|].
\end{align*}}
\noindent Here, the last inequality follows from $\bv_k=G_k^T[G_kG_k^T]^{-1}c_k$ and Assumption \ref{ass:1-1}.
\end{proof}

We further simplify the result of \eqref{eq:ared_k} using the trust-region scheme in \eqref{Ful_GenerateRadius}.

\begin{lemma}\label{lemma:Ful_cond2}

Suppose Assumptions \ref{ass:1-1}, \ref{ass:stabilized_penalty}, and \ref{ass:Ful_unbias} hold and $\{\beta_k\}\subseteq(0,\beta_{\max}]$.~For all $k\geq\barK+1$, we have
\begin{multline*}
\mE_k[\L_{\barmu_{\barK}}^{k+1}]
\leq  \L_{\barmu_{\barK}}^k-\frac{1}{4}\eta_{2,k}\alpha_k\|\nabla\L_k\|^2 + \frac{\delta\kappa_c}{\sqrt{\kappa_{1,G}}} \alpha_k^2\mE_k[\|\nabla f_k-\barg_k\|] \\
+ \left(\zeta+\eta_{1,k}\tau_k\right)\eta_{1,k}\alpha_k^2\mE_k[\|\nabla f_k-\barg_k\|^2].
\end{multline*}
\end{lemma}

\begin{proof}
According to the definition in \eqref{Ful_GenerateRadius}, we separate the proof into the following three cases: $\|\bar{\nabla}\L_k\|\in(0,1/\eta_{1,k})$, $\|\bar{\nabla}\L_k\|\in[1/\eta_{1,k},1/\eta_{2,k}]$, and $\|\bar{\nabla}\L_k\|\in(1/\eta_{2,k}, \infty)$.

\vskip4pt

\noindent\textbf{Case 1}, $\|\bar{\nabla}\L_k\|\in(0,1/\eta_{1,k})$. We have $\Delta_k = \eta_{1,k}\alpha_k\|\bnabla\mL_k\|$, therefore
\begin{align*}
-\|\bar{\nabla}\L_k\|\Delta_k+\frac{1}{2}\|B_k\|\Delta_k^2 & = -\eta_{1,k}\alpha_k\|\bar{\nabla}\L_k\|^2+\frac{1}{2}\eta_{1,k}^2\alpha_k^2\|B_k\|\|\bar{\nabla}\L_k\|^2\notag\\
& = -\left(1-\frac{1}{2}\eta_{1,k}\alpha_k\|B_k\|\right)\eta_{1,k}\alpha_k\|\bar{\nabla}\L_k\|^2.
\end{align*}
\noindent Plugging the above expression into \eqref{eq:ared_k} and applying \eqref{Ful_GenerateRadius}, we have
\begin{align}\label{eq:bddvar_case1}
\L_{\barmu_{\barK}}^{k+1}-\L_{\barmu_{\barK}}^k & \leq
-\left(1-\frac{1}{2}\eta_{1,k}\alpha_k\|B_k\|\right)\eta_{1,k}\alpha_k\|\bar{\nabla}\L_k\|^2+\bargamma_k(\nabla f_k-\barg_k)^T\bv_k\notag\\
&\quad  +  \eta_{1,k}\alpha_k\|P_k(\nabla f_k-\barg_k)\|\|\bar{\nabla}\L_k\|+\frac{1}{2}\eta_{1,k}^2\alpha_k^2\tau_k\|\bar{\nabla}\L_k\|^2\notag\\
& \leq -\frac{1}{2}\left(1-\eta_{1,k}\alpha_k\|B_k\|-\eta_{1,k}\alpha_k\tau_k\right)\eta_{1,k}\alpha_k\|\bar{\nabla}\L_k\|^2\notag\\
& \quad +\bargamma_k(\nabla f_k-\barg_k)^T\bv_k +\frac{1}{2}\eta_{1,k}\alpha_k\|P_k(\nabla f_k-\barg_k)\|^2\;(\text{by Young's inequality})\notag\\
& \leq -\left(\frac{1}{2}-\eta_{1,k}\alpha_k\tau_k\right)\eta_{1,k}\alpha_k\|\bar{\nabla}\L_k\|^2+\bargamma_k(\nabla f_k-\barg_k)^T\bv_k \notag\\
& \quad +\frac{1}{2}\eta_{1,k}\alpha_k\|P_k(\nabla f_k-\barg_k)\|^2 \quad(\text{since by \eqref{def:eta2k}, }\|B_k\|\leq \tau_k).
\end{align}
\noindent\textbf{Case 2}, $\|\bar{\nabla}\L_k\|\in[1/\eta_{1,k},1/\eta_{2,k}]$. We have $\Delta_k=\alpha_k$ and thus 
\begin{align*}
- \|\bar{\nabla}\L_k\|\Delta_k +\frac{1}{2}\|B_k\|\Delta_k^2 
 = - \|\bar{\nabla}\L_k\|\alpha_k +\frac{1}{2}\|B_k\|\alpha_k^2  \leq -\eta_{2,k}\alpha_k\|\bar{\nabla}\L_k\|^2+\frac{1}{2}\eta_{1,k}^2\alpha_k^2\|B_k\|\|\bar{\nabla}\L_k\|^2,
\end{align*}
\noindent where the inequality is due to $\eta_{1,k}\|\bar{\nabla}\L_k\|\geq 1\geq \eta_{2,k}\|\bar{\nabla}\L_k\|$. Plugging the above~expression into \eqref{eq:ared_k}, using the relation $\eta_{1,k}\|\bar{\nabla}\L_k\|\geq 1$ again, we have {\allowdisplaybreaks
\begin{align}\label{pequ:case2}
\L_{\barmu_{\barK}}^{k+1}-\L_{\barmu_{\barK}}^k & \leq -\left(\eta_{2,k}-\frac{1}{2}\eta_{1,k}^2\alpha_k\|B_k\|\right)\alpha_k\|\bar{\nabla}\L_k\|^2+\bargamma_k(\nabla f_k-\barg_k)^T\bv_k \notag\\
& \quad +\eta_{1,k}\alpha_k\|P_k(\nabla f_k-\barg_k)\|\|\bar{\nabla}\L_k\|+\frac{1}{2}\eta_{1,k}^2\alpha_k^2\tau_k\|\bar{\nabla}\L_k\|^2 \notag\\
& \leq -\left(\eta_{2,k}-\frac{1}{2}\eta_{1,k}^2\alpha_k\|B_k\|-\frac{1}{2}\eta_{1,k}-\frac{1}{2}\eta_{1,k}^2\alpha_k\tau_k\right)\alpha_k\|\bar{\nabla}\L_k\|^2 \notag\\
& \quad +\bargamma_k(\nabla f_k-\barg_k)^T\bv_k+\frac{1}{2}\eta_{1,k}\alpha_k\|P_k(\nabla f_k-\barg_k)\|^2 \;(\text{by Young's inequality})\notag\\
& \leq -\left(\eta_{2,k}-\frac{1}{2}\eta_{1,k}-\eta_{1,k}^2\alpha_k\tau_k\right)\alpha_k\|\bar{\nabla}\L_k\|^2 +\bargamma_k(\nabla f_k-\barg_k)^T\bv_k \notag\\
&\quad +\frac{1}{2}\eta_{1,k}\alpha_k\|P_k(\nabla f_k-\barg_k)\|^2\quad(\text{since by \eqref{def:eta2k}, }\|B_k\|\leq \tau_k).
\end{align}}
\noindent\textbf{Case 3}, $\|\bar{\nabla}\L_k\|\in(1/\eta_{2,k},\infty)$. We have $\Delta_k=\eta_{2,k}\alpha_k\|\bar{\nabla}\L_k\|$ and thus 
\begin{align*}
-\|\bar{\nabla}\L_k\|\Delta_k+\frac{1}{2}\|B_k\|\Delta_k^2  & = -\eta_{2,k}\alpha_k\|\bar{\nabla}\L_k\|^2+\frac{1}{2}\eta_{2,k}^2\alpha_k^2\|B_k\|\|\bar{\nabla}\L_k\|^2 \notag\\
& = -\left(1-\frac{1}{2}\eta_{2,k}\alpha_k\|B_k\|\right)\eta_{2,k}\alpha_k\|\bar{\nabla}\L_k\|^2.
\end{align*}
\noindent Plugging into \eqref{eq:ared_k} and applying \eqref{Ful_GenerateRadius}, we have 
\begin{align}\label{pequ:case3}
\L_{\barmu_{\barK}}^{k+1}-\L_{\barmu_{\barK}}^k & \leq -\left(1-\frac{1}{2}\eta_{2,k}\alpha_k\|B_k\|\right)\eta_{2,k}\alpha_k\|\bar{\nabla}\L_k\|^2 +\bargamma_k(\nabla f_k-\barg_k)^T\bv_k \notag\\
&\quad  + \eta_{2,k}\alpha_k\|P_k(\nabla f_k-\barg_k)\|\|\bar{\nabla}\L_k\|  +\frac{1}{2}\eta_{2,k}^2\alpha_k^2\tau_k\|\bar{\nabla}\L_k\|^2 \notag\\
& \leq -\frac{1}{2}\left(1-\eta_{2,k}\alpha_k\|B_k\|-\eta_{2,k}\alpha_k\tau_k\right)\eta_{2,k}\alpha_k\|\bar{\nabla}\L_k\|^2\notag\\
& \quad +\bargamma_k(\nabla f_k-\barg_k)^T\bv_k +\frac{1}{2}\eta_{2,k}\alpha_k\|P_k(\nabla f_k-\barg_k)\|^2\;(\text{by Young's inequality})\notag\\
& \leq -\left(\frac{1}{2}-\eta_{2,k}\alpha_k\tau_k\right)\eta_{2,k}\alpha_k\|\bar{\nabla}\L_k\|^2+\bargamma_k(\nabla f_k-\barg_k)^T\bv_k\notag\\
& \quad  +\frac{1}{2}\eta_{2,k}\alpha_k\|P_k(\nabla f_k-\barg_k)\|^2 \quad(\text{since by \eqref{def:eta2k}, }\|B_k\|\leq \tau_k).
\end{align}
\noindent Using $\eta_{2,k}\leq \eta_{1,k}$ and taking an upper for the results of the three cases in \eqref{eq:bddvar_case1}, \eqref{pequ:case2}, and \eqref{pequ:case3}, we have 
\begin{multline*}
\L_{\barmu_{\barK}}^{k+1} -\L_{\barmu_{\barK}}^k  \leq -\left(\eta_{2,k}-\frac{1}{2}\eta_{1,k}-\eta_{1,k}^2\alpha_k\tau_k\right)\alpha_k\|\bar{\nabla}\L_k\|^2 \notag\\
+\bargamma_k(\nabla f_k-\barg_k)^T\bv_k+\frac{1}{2}\eta_{1,k}\alpha_k\|P_k(\nabla f_k-\barg_k)\|^2.
\end{multline*}
\noindent Taking expectation conditional on $\bx_k$, applying Lemma \ref{lemma:exp<=0}, and noting that $\mE_k[\|\bar{\nabla}\L_k\|^2] = \|\nabla\L_k\|^2 + \mE_k[\|P_k(\nabla f_k-\barg_k)\|^2]$, we have
{\allowdisplaybreaks\begin{align*}
\mE_k[\L_{\barmu_{\barK}}^{k+1}]-\L_{\barmu_{\barK}}^k 
& \leq -\left(\eta_{2,k}-\frac{1}{2}\eta_{1,k}-\eta_{1,k}^2\alpha_k\tau_k\right)\alpha_k\mE_k[\|\bar{\nabla}\L_k\|^2] \\
& \quad +\frac{\delta\kappa_c}{\sqrt{\kappa_{1,G}}}\alpha_k^2\mE_k[\|\nabla f_k-\barg_k\|] +\frac{1}{2}\eta_{1,k}\alpha_k\mE_k[\|P_k(\nabla f_k-\barg_k)\|^2]\\
& = -\left(\eta_{2,k}-\frac{1}{2}\eta_{1,k}-\eta_{1,k}^2\alpha_k\tau_k\right)\alpha_k\|\nabla\L_k\|^2\\
&\quad -\left(\eta_{2,k}-\frac{1}{2}\eta_{1,k}-\eta_{1,k}^2\alpha_k\tau_k\right)\alpha_k\mE_k[\|P_k(\nabla f_k-\barg_k)\|^2]\\
&\quad +\frac{\delta\kappa_c}{\sqrt{\kappa_{1,G}}}\alpha_k^2\mE_k[\|\nabla f_k-\barg_k\|] +\frac{1}{2}\eta_{1,k}\alpha_k\mE_k[\|P_k(\nabla f_k-\barg_k)\|^2]\\
& = -\left(\eta_{2,k}-\frac{1}{2}\eta_{1,k}-\eta_{1,k}^2\alpha_k\tau_k\right)\alpha_k\|\nabla\L_k\|^2+\frac{\delta\kappa_c}{\sqrt{\kappa_{1,G}}}\alpha_k^2\mE_k[\|\nabla f_k-\barg_k\|]\\
&\quad +\left(\eta_{1,k}-\eta_{2,k}+\eta_{1,k}^2\alpha_k\tau_k\right)\alpha_k\mE_k[\|P_k(\nabla f_k-\barg_k)\|^2].
\end{align*}}
\noindent Furthermore, we note that 
\begin{align*}
\alpha_k \stackrel{\eqref{def:eta2k}}{\leq} \frac{2}{8\eta_{1,k}\tau_k+3\zeta} & \Longrightarrow 3\zeta\alpha_k+8\eta_{1,k}\alpha_k\tau_k\leq 2 \\
& \Longrightarrow \frac{1}{2}+\eta_{1,k}\alpha_k\tau_k\leq \frac{3}{4}-\frac{3}{8}\zeta\alpha_k \\
& \Longrightarrow \frac{1}{2}\eta_{1,k}+\eta_{1,k}^2\alpha_k\tau_k\leq \frac{3}{4}\eta_{1,k}\left(1-\frac{1}{2}\zeta\alpha_k\right)\stackrel{\eqref{def:eta2k}}{=} \frac{3}{4}\eta_{2,k}\\
& \Longrightarrow -\left(\eta_{2,k}-\frac{1}{2}\eta_{1,k}-\eta_{1,k}^2\alpha_k\tau_k\right) \leq -\frac{1}{4}\eta_{2,k}.
\end{align*}
\noindent Combining the above two results and using \eqref{def:eta2k}, we have
\begin{multline*}
\mE_k[ \L_{\barmu_{\barK}}^{k+1}]-\L_{\barmu_{\barK}}^k \leq -\frac{1}{4}\eta_{2,k}\alpha_k\|\nabla\L_k\|^2+ \frac{\delta\kappa_c}{\sqrt{\kappa_{1,G}}}\alpha_k^2\mE_k[\|\nabla f_k-\barg_k\|]\\
+\left(\zeta+\eta_{1,k}\tau_k\right)\eta_{1,k}\alpha_k^2\mE_k[\|P_k(\nabla f_k-\barg_k)\|^2].
\end{multline*}
\noindent The conclusion follows by noting that $\mE_k[\|P_k(\nabla f_k-\barg_k)\|^2]\leq\mE_k[\|\nabla f_k-\barg_k\|^2]$.
\end{proof}

Finally, we present some properties of the control parameters generated in Step~1 of Algorithm~\ref{Alg:Non-adap}.

 \begin{lemma}\label{lemma:4statements}

Let Assumptions \ref{ass:1-1}, \ref{ass:stabilized_penalty} hold and $\{\beta_k\}\subseteq(0,\beta_{\max}]$. For all~$k\geq0$,\\
(a) there exist constants $\eta_{\min},\eta_{\max}>0$ such that $\eta_{\min}\leq\eta_{2,k}\leq\eta_{1,k}\leq\eta_{\max}$;\\
(b) there exists a constant $\tau_{\max}>0$ such that $\tau_k\leq\tau_{\max}$;\\
(c) there exist constants $\alpha_l,\alpha_u>0$ such that $\alpha_k\in[\alpha_l\beta_k,\alpha_u\beta_k]$.
\end{lemma}

\begin{proof}
(a) By \eqref{def:eta2k}, we see that $\eta_{2,k}\leq \eta_{1,k}$. Further, by Assumption \ref{ass:1-1}, we have
\begin{align*}
\eta_{1,k} & \stackrel{\eqref{def:eta2k}}{=} \zeta\cdot\|\bv_k\|/\|c_k\|\leq \zeta\cdot\|G_k^T[G_kG_k^T]^{-1}\|\leq \zeta/\sqrt{\kappa_{1,G}} \eqqcolon \eta_{\max},\\
\eta_{2,k} & \stackrel{\eqref{def:eta2k}}{=} \eta_{1,k}\rbr{1 - \frac{\zeta\alpha_k}{2}} \stackrel{\eqref{def:eta2k}}{\geq}\eta_{1,k}\rbr{1 - \frac{\zeta}{2}\cdot \frac{1}{4\zeta}} \stackrel{\eqref{def:eta2k}}{\geq} \frac{7\zeta\|\bv_k\|}{8\|c_k\|}\geq \frac{7\zeta }{8\sqrt{\kappa_{2,G}}}\eqqcolon\eta_{\min}.
\end{align*}
\noindent (b) By Assumptions \ref{ass:1-1} and \ref{ass:stabilized_penalty}, we have $L_{\nabla f,k}\leq L_{\nabla f}, L_{G,k}\leq L_G$, $\|B_k\|\leq\kappa_B$, and $\barmu_{k}\leq \hat{\mu}$. Thus, we let $\tau_{\max} \coloneqq L_{\nabla f}+L_G\hatmu+\kappa_B$ and the result holds. (c)~We~let~$\alpha_l \coloneqq1 / (4\eta_{\max}\tau_{\max}\beta_{\max}+4\zeta\beta_{\max})$ and $\alpha_u \coloneqq 1/(4\zeta\beta_{\max})$, and the result holds.
 \end{proof}

In the next subsection, we use Lemmas \ref{lemma:Ful_cond2} and \ref{lemma:4statements} to show the global convergence of TR-StoSQP. We consider both constant and decaying $\beta_k$ sequences.

%% file: sec4_global_convergence.tex
\subsection{Global convergence}\label{subsec:global_conv}

We first consider constant $\beta_k$, i.e., $\beta_k=\beta\in(0, \beta_{\max}]$, $\forall k\geq 0$. We show that the expectation of weighted averaged KKT residuals converges to a neighborhood around zero with a radius of the order $\cO(\beta)$. When the growth~condition parameter $M_{g,1}=0$ (cf. Assumption \ref{ass:Ful_unbias}), the weighted average reduces to the uniform average. 

\begin{lemma}\label{thm:Ful_bddvar_const}

Suppose Assumptions \ref{ass:1-1}, \ref{ass:stabilized_penalty}, and \ref{ass:Ful_unbias} hold and $\beta_k=\beta\in(0,\beta_{\max}]$, $\forall k\geq 0$. For any positive integer $K>0$, we define $w_k = (1+\Upsilon M_{g,1}\beta^2)^{\barK+K-k}$,~$\barK\leq k\leq \barK+K$, with $\Upsilon \coloneqq (\zeta\eta_{\max}+\eta_{\max}^2\tau_{\max}+\delta\kappa_c/\sqrt{\kappa_{1,G}})\alpha_u^2$. We have (cf. $\mE_{\barK+1}[\cdot] = \mE[\cdot\mid \mF_{\barK}]$)
\begin{equation*}
\mE_{\barK+1}\left[\frac{\sum_{k=\barK+1}^{\barK+K}w_k\|\nabla\L_k\|^2}{\sum_{k=\barK+1}^{\barK+K}w_k}\right] \leq\frac{4}{\eta_{\min}\alpha_l\beta}\cdot\frac{w_{\barK}(\L_{\barmu_{\barK}}^{\barK+1}-f_{\inf})}{\sum_{k=\barK+1}^{\barK+K}w_k}+\frac{4\Upsilon M_g}{ \eta_{\min}\alpha_l}\beta.
\end{equation*}

\end{lemma}

\begin{proof}
From Lemma \ref{lemma:Ful_cond2} and Assumption \ref{ass:Ful_unbias}, we have for any $k\geq \barK+1$,
\begin{align*}
&\mE_k[\L_{\barmu_{\barK}}^{k+1}] \leq \L_{\barmu_{\barK}}^k -\frac{1}{4}\eta_{2,k}\alpha_k\|\nabla\L_k\|^2 +\left(\zeta+\eta_{1,k}\tau_k\right)\eta_{1,k}\alpha_k^2[M_g+M_{g,1}(f_k-f_{\inf})]\\
&\qquad\quad\qquad+\frac{\delta\kappa_c}{\sqrt{\kappa_{1,G}}}\alpha_k^2\sqrt{M_g+M_{g,1}(f_k-f_{\inf})}\\
& \hskip0.5cm \stackrel{\mathclap{\text{Lemma \ref{lemma:4statements}} } }{\leq} \;\;\; \L_{\barmu_{\barK}}^k-\frac{1}{4}\eta_{\min}\alpha_l\beta\|\nabla\L_k\|^2+\Upsilon\beta^2[M_g+M_{g,1}(f_k-f_{\inf})] \quad (\text{by } M_g\geq 1).
\end{align*}
\noindent Using the fact that $f_k-f_{\inf}\leq f_k-f_{\inf}+\barmu_{\barK}\|c_k\|=\L_{\barmu_{\barK}}^k-f_{\inf}$, we obtain 
\begin{equation*}
\hskip1cm \mE_k[\L_{\barmu_{\barK}}^{k+1}-f_{\inf}]  \leq  \left(1+\Upsilon M_{g,1}\beta^2\right)(\L_{\barmu_{\barK}}^k-f_{\inf}) -\frac{1}{4} \eta_{\min}\alpha_l \beta\|\nabla\L_k\|^2+\Upsilon M_g\beta^2.
\end{equation*}
\noindent Taking the expectation conditional on $\mF_{\barK}$ and rearranging the terms, we have 
\begin{equation*}
\mE_{\barK+1}[\|\nabla\L_k\|^2] \leq \frac{4(1+\Upsilon M_{g,1}\beta^2)}{\eta_{\min}\alpha_l\beta}\mE_{\barK+1}[\L_{\barmu_{\barK}}^k-f_{\inf}] - \frac{4}{\eta_{\min}\alpha_l\beta}\mE_{\barK+1}[\L_{\barmu_{\barK}}^{k+1}-f_{\inf}]+\frac{4 \Upsilon M_g}{\eta_{\min}\alpha_l}\beta.
\end{equation*}
\noindent Multiplying $w_k$ on both sides and summing over $k=\barK+1,\cdots,\barK+K$, we have 
\begin{multline*}
\mE_{\barK+1} \left[\frac{\sum_{k=\barK+1}^{\barK+K}w_k\|\nabla\L_k\|^2}{\sum_{k=\barK+1}^{\barK+K}w_k} \right] = \frac{\sum_{k=\barK+1}^{\barK+K}w_k\mE_{\barK+1}[\|\nabla\L_k\|^2]}{\sum_{k=\barK+1}^{\barK+K}w_k} \\
\leq \frac{4}{\eta_{\min}\alpha_l\beta}\cdot\frac{w_{\barK}(\L_{\barmu_{\barK}}^{\barK+1}-f_{\inf})-\mE_{\barK+1}[\L_{\barmu_{\barK}}^{\barK+K+1}-f_{\inf}]}{\sum_{k=\barK+1}^{\barK+K}w_k}+\frac{4\Upsilon M_{g}}{\eta_{\min}\alpha_l}\beta,
\end{multline*} 
\noindent where the first equality uses the fact that $\barK$ is fixed in the conditional expectation. Noting that $\mE_{\barK+1}[\L_{\barmu_{\barK}}^{\barK+K+1}-f_{\inf}]\geq 0$, we complete the proof.
\end{proof}

The following theorem follows from Lemma \ref{thm:Ful_bddvar_const}.

\begin{theorem}[Global convergence with constant $\beta_k$]\label{cor:cons_beta}
Suppose Assumptions~\ref{ass:1-1}, \ref{ass:stabilized_penalty}, and \ref{ass:Ful_unbias} hold and $\beta_k=\beta\in(0,\beta_{\max}]$, $\forall k\geq 0$. Let us define $w_k$ and $\Upsilon$ as in Lemma \ref{thm:Ful_bddvar_const}. We have\\
(a) when $M_{g,1}=0$, 
\begin{equation*}
\lim_{K\rightarrow\infty}\mE\left[\frac{1}{K}\sum_{k=\barK+1}^{\barK+K}\|\nabla\L_k\|^2\right]\leq \frac{4\Upsilon M_g}{\eta_{\min}\alpha_l}\beta;
\end{equation*} 
\noindent (b) when $M_{g,1}>0$, 
\begin{equation*}
\lim_{K\rightarrow\infty}\mE\left[\frac{1}{\sum_{k=\barK+1}^{\barK+K}w_k}\sum_{k=\barK+1}^{\barK+K}w_k\|\nabla\L_k\|^2\right] \leq\frac{4\Upsilon \{M_{g,1}\mE[\L_{\barmu_{\barK}}^{\barK+1}-f_{\inf}]+M_g\}}{\eta_{\min}\alpha_l}\beta. 
\end{equation*}		
\end{theorem}

\begin{proof}
(a) When $M_{g,1}=0$, we have $w_k=1$ for $\barK+1 \leq k\leq\barK+K$. From~Lemma \ref{thm:Ful_bddvar_const}, we have  
\begin{equation*}
\mE_{\barK+1}\left[\frac{1}{K}\sum_{k=\barK+1}^{\barK+K}\|\nabla\L_k\|^2\right]
\leq\frac{4}{\eta_{\min}\alpha_l\beta}\cdot\frac{\L_{\barmu_{\barK}}^{\barK+1}-f_{\inf}}{K}+\frac{4\Upsilon M_g}{\eta_{\min}\alpha_l}\beta.
\end{equation*}
Letting $K\rightarrow\infty$ and using the fact that $\|\nabla\L_k\|^2\leq \kappa_{\nabla f}^2+\kappa_c^2$ (cf. Assumption~\ref{ass:1-1}), we apply Fatou's lemma and have (the $\lim$ on the left can be strengthened to $\limsup$)  
\begin{equation*}
\lim_{K\rightarrow\infty}\mE\left[\frac{1}{K}\sum_{k=\barK+1}^{\barK+K}\|\nabla\L_k\|^2\right] \leq \mE\sbr{\limsup_{K\rightarrow\infty}\mE_{\barK+1}\left[\frac{1}{K}\sum_{k=\barK+1}^{\barK+K}\|\nabla\L_k\|^2\right] }\leq \frac{4\Upsilon M_g}{\eta_{\min}\alpha_l}\beta.
\end{equation*} 
\noindent(b) When $M_{g,1}>0$, we apply Lemma \ref{thm:Ful_bddvar_const} and the fact that $\sum_{k=\barK+1}^{\barK+K}w_k=(w_{\barK}-1)/(\Upsilon M_{g,1}\beta^2)$, and obtain  
\begin{equation*}
\mE_{\barK+1}\left[\frac{\sum_{k=\barK+1}^{\barK+K}w_k\|\nabla\L_k\|^2}{\sum_{k=\barK+1}^{\barK+K}w_k}\right]
\leq\frac{4\Upsilon M_{g,1}\beta}{\eta_{\min}\alpha_l}\cdot\frac{w_{\barK}(\L_{\barmu_{\barK}}^{\barK+1}-f_{\inf})}{w_{\barK}-1}+\frac{4\Upsilon M_g}{\eta_{\min}\alpha_l}\beta.
\end{equation*}
Since $w_{\barK}/(w_{\barK}-1) = (1+\Upsilon M_{g,1}\beta^2)^K/\{(1+\Upsilon M_{g,1}\beta^2)^K-1\}\rightarrow 1$ as $K\rightarrow\infty$, we apply Fatou's lemma and have (the $\lim$ on the left can be strengthened to $\limsup$)
\begin{multline*}
\lim_{K\rightarrow\infty}\mE\left[\frac{\sum_{k=\barK+1}^{\barK+K}w_k\|\nabla\L_k\|^2}{\sum_{k=\barK+1}^{\barK+K}w_k}\right]\leq \mE\sbr{\limsup_{K\rightarrow\infty} \mE_{\barK+1}\left[\frac{\sum_{k=\barK+1}^{\barK+K}w_k\|\nabla\L_k\|^2}{\sum_{k=\barK+1}^{\barK+K}w_k}\right]}\\
\leq \frac{4\Upsilon \{M_{g,1}\mE[\L_{\barmu_{\barK}}^{\barK+1}-f_{\inf}]+M_g\}}{\eta_{\min}\alpha_l}\beta.
\end{multline*} 
\noindent This completes the proof.
\end{proof}

From Theorem \ref{cor:cons_beta}, we note that the radius of the local neighborhood is proportional to $\beta$. Thus, to decrease the radius, one should choose a smaller $\beta$. However,~the trust-region radius is also proportional to $\beta$ (cf. \eqref{Ful_GenerateRadius}); thus, a smaller $\beta$ may result in a slow convergence. This suggests the existence of a trade-off between the convergence speed and convergence precision. 

For constant $\{\beta_k\}$, \cite{Berahas2021Stochastic,Berahas2021Sequential, Curtis2021Inexact, Curtis2020fully} established similar global results to Theorem \ref{cor:cons_beta}. However, our analysis has two major differences. (i) That line of literature~\mbox{required}~$\beta$~to be upper bounded by some complex quantities that may be less than 1, while we do not need such a condition. (ii) Compared to the stochastic trust-region method~for~unconstrained optimization \citep{Curtis2020fully}, our local neighborhood radius is proportional to the~input $\beta$ (i.e., we can control the radius by the input), while the one in \cite{Curtis2020fully} is independent~of~$\beta$.

Next, we consider decaying $\beta_k$. We show in the next lemma that, when $\sum\beta_k=\infty$ and $\sum\beta_k^2<\infty$, the infimum of KKT residuals converges to zero almost surely. Based on this result, we further show that the KKT residuals converge to zero almost surely.

\begin{lemma}\label{thm:Ful_bddvar_liminf}
Suppose Assumptions \ref{ass:1-1}, \ref{ass:stabilized_penalty}, and \ref{ass:Ful_unbias} hold, $\{\beta_k\}\subseteq (0,\beta_{\max}]$, and $\sum_{k=0}^{\infty}\beta_k=\infty$ and $\sum_{k=0}^{\infty}\beta_k^2<\infty$. We have  
\begin{equation*}
\liminf_{k\rightarrow\infty}\|\nabla \L_k\|=0\quad\text{almost surely}.
\end{equation*}	
\end{lemma}

\begin{proof}
From the proof of Lemma \ref{thm:Ful_bddvar_const}, we have for any $\forall k\geq \barK+1$ that 
\begin{equation*}
\mE_k[\L_{\barmu_{\barK}}^{k+1}-f_{\inf}] \leq  \left(1+\Upsilon M_{g,1}\beta_k^2\right)(\L_{\barmu_{\barK}}^k-f_{\inf})-\frac{1}{4}\eta_{\min}\alpha_l\beta_k\|\nabla\L_k\|^2+\Upsilon M_g\beta_k^2.
\end{equation*}
Since $\L_{\barmu}(\bx)-f_{\inf}$ is bounded below by zero, $\eta_{\min}\alpha_l\beta_k\|\nabla\L_k\|^2>0$, and $\sum_{k=\barK+1}^\infty \beta_k^2<\infty$, it immediately follows from the Robbins-Siegmund theorem \citep{Robbins1971Convergence} that
\begin{equation}\label{pequ:4}
\sup_{k\geq\barK+1}\mE_{\barK+1}[\L_{\barmu_{\barK}}^k-f_{\inf}]\coloneqq M_{\barK}<\infty,\;\;\;\;\quad \sum_{k=\barK+1}^{\infty}\beta_k\mE_{\barK+1}[\|\nabla \L_k\|^2]  < \infty.
\end{equation}
\noindent The latter part suggests that $P[\sum_{k=\barK+1}^{\infty}\beta_k\|\nabla\mL_k\|^2<\infty \mid \mF_{\barK}]=1$. Since the result holds for any $\mF_{\barK}$, we have $P[\sum_{k=\barK+1}^{\infty}\beta_k\|\nabla\mL_k\|^2<\infty]=1$. Noting that $\sum_{k=\barK+1}^{\infty}\beta_k =\infty$ for any run of the algorithm, we complete the proof.
\end{proof}

Finally, we establish the global convergence theorem for decaying $\beta_k$ sequence.

\begin{theorem}[Global convergence with decaying $\beta_k$]\label{thm:Ful_bddvar_limit}
Suppose Assumptions~\ref{ass:1-1}, \ref{ass:stabilized_penalty}, and \ref{ass:Ful_unbias} hold, $\{\beta_k\}\subseteq (0,\beta_{\max}]$, and $\sum_{k=0}^{\infty}\beta_k=\infty$ and $\sum_{k=0}^{\infty}\beta_k^2<\infty$. We have
\begin{equation*}
\lim_{k\rightarrow\infty}\|\nabla \L_k\|=0\quad\text{almost surely}.
\end{equation*}
\end{theorem}

\begin{proof}
For any run of the algorithm, suppose the statement does not hold, then~we have $\limsup_{k\rightarrow\infty}\|\nabla\L_k\|\geq 2\epsilon$ for some $\epsilon>0$. For such a run, let us define the~set~$\K_{\epsilon} \coloneqq \{k\geq \barK+1: \|\nabla\L_k\|\geq \epsilon\}$. By Lemma \ref{thm:Ful_bddvar_liminf}, there exist two infinite index sets $\{m_i\}$, $\{n_i\}$ with $\barK< m_i<n_i$, $\forall i\geq 0$, such that 
\begin{equation}\label{pequ:5}
\|\nabla \L_{m_i}\|\geq 2 \epsilon,\quad  \|\nabla \L_{n_i}\|<\epsilon, \quad  \|\nabla \L_k\|\geq\epsilon \text{ for } k\in\{m_i+1,\cdots,n_i-1\}.
\end{equation}  
\noindent By Assumption \ref{ass:1-1} and the definition $\nabla \L_k = (P_k\nabla f_k, c_k)$, there exists $L_{\nabla \L}>0$ such that $\|\nabla\L_{k+1}-\nabla \L_k\|\leq L_{\nabla\L}\{\|\bx_{k+1}-\bx_k\| + \|\bx_{k+1}-\bx_k\|^2\}$. Thus, \eqref{pequ:5} implies  
\begin{align*}
\epsilon & \leq \|\nabla\L_{m_i}\| - \|\nabla\L_{n_i}\| \leq \| \nabla\L_{n_i} - \nabla\L_{m_i}\| \leq \sum_{k=m_i}^{n_i-1}\|\nabla\L_{k+1} - \nabla\L_k\|\\
& \leq L_{\nabla\L}\sum_{k=m_i}^{n_i-1} \{\|\bx_{k+1}-\bx_k\| + \|\bx_{k+1}-\bx_k\|^2\}\leq L_{\nabla\L}\sum_{k=m_i}^{n_i-1} (\Delta_k + \Delta_k^2)\\
& \stackrel{\mathclap{\eqref{Ful_GenerateRadius}}}{\leq} L_{\nabla\L}\sum_{k=m_i}^{n_i-1} (\eta_{\max}\alpha_u\beta_k\|\bar{\nabla}\L_k\| + \eta_{\max}^2\alpha_u^2\beta_k^2\|\bar{\nabla}\L_k\|^2)\quad(\text{also by Lemma }\ref{lemma:4statements}).
\end{align*} 
\noindent Since $\|\bar{\nabla}\L_k\|\leq \|\nabla\L_k\| + \|\barg_k - \nabla f_k\|$, $\|\bar{\nabla}\L_k\|^2 \leq 2(\|\nabla\L_k\|^2 + \|\barg_k - \nabla f_k\|^2)$ and~$\beta_k\leq\beta_{\max}$, we have  
\begin{align*}
\epsilon & \leq L_{\nabla\L}\eta_{\max}\alpha_u\sum_{k=m_i}^{n_i-1} \beta_k\|\nabla\L_k\|+ 2L_{\nabla\L}\eta_{\max}^2\alpha_u^2\beta_{\max}\sum_{k=m_i}^{n_i-1} \beta_k\|\nabla\L_k\|^2\\
&\quad + L_{\nabla\L}\eta_{\max}\alpha_u\sum_{k=m_i}^{n_i-1} \beta_k\|\barg_k - \nabla f_k\| + 2L_{\nabla\L}\eta_{\max}^2\alpha_u^2\beta_{\max}\sum_{k=m_i}^{n_i-1}\beta_k \|\barg_k-\nabla f_k\|^2.
\end{align*} 
\noindent Multiplying $\epsilon$ on both sides and using $\|\nabla \L_k\|\geq\epsilon$ for $k\in\{m_i,\cdots,n_i-1\}$, we have  
\begin{multline}\label{pequ:6}
\epsilon^2 \leq \{L_{\nabla\L}\eta_{\max}\alpha_u+ 2\epsilon L_{\nabla\L}\eta_{\max}^2\alpha_u^2\beta_{\max}\}\sum_{k=m_i}^{n_i-1} \beta_k\|\nabla\L_k\|^2 \\
+ \cbr{\epsilon L_{\nabla\L}\eta_{\max}\alpha_u + 2\epsilon L_{\nabla\L}\eta_{\max}^2\alpha_u^2\beta_{\max}}\sum_{k=m_i}^{n_i-1} \beta_k\rbr{\|\barg_k - \nabla f_k\| + \|\barg_k-\nabla f_k\|^2}.
\end{multline} 
\noindent For sake of contradiction, we will show that the right-hand-side of the above expression converges to zero as $i\rightarrow\infty$. By \eqref{pequ:4}, we know that $\infty > \sum_{k=\barK+1}^{\infty}\beta_k\|\nabla\L_k\|^2 \geq \sum_{i=0}^\infty\sum_{k=m_i}^{n_i-1} \beta_k\|\nabla\L_k\|^2$. Thus, $\sum_{k=m_i}^{n_i-1} \beta_k\|\nabla\L_k\|^2\rightarrow 0$ as $i\rightarrow\infty$. For the \mbox{second}~term, we note that 
{\allowdisplaybreaks\begin{align*}
&\sum_{i=0}^{\infty} \mE_{\barK+1}  \left[\sum_{k=m_i}^{n_i-1} \beta_k(\|\barg_k -  \nabla f_k\|+\|\barg_k -  \nabla f_k\|^2)\right] \\
& \qquad  = \sum_{i=0}^{\infty}\sum_{k=m_i}^{n_i-1} \beta_k\mE_{\barK+1}[\|\barg_k - \nabla f_k\| + \|\barg_k - \nabla f_k\|^2]\\
& \qquad \leq 2\sum_{i= 0}^{\infty}\sum_{k=m_i}^{n_i-1} \beta_k(M_g+M_{g,1}\mE_{\barK+1}[f_k-f_{\inf}]) \stackrel{\mathclap{\eqref{pequ:4}}}{\leq }2(M_g+M_{g,1}M_{\barK})\sum_{i= 0}^{\infty}\sum_{k=m_i}^{n_i-1} \beta_k.
\end{align*}} 
\noindent By the definition of $\K_{\epsilon}$ and \eqref{pequ:4}, we have $\sum_{i= 0}^{\infty}\sum_{k=m_i}^{n_i-1} \beta_k\leq \sum_{k\in \K_{\epsilon}}\beta_k<\infty$.~We apply Borel-Cantelli lemma, integrate out the randomness of $\mF_{\barK}$, and have $\sum_{k=m_i}^{n_i-1} \beta_k(\|\barg_k - \nabla f_k\|+\|\barg_k - \nabla f_k\|^2)\rightarrow 0$ as $i\rightarrow\infty$ almost surely. Thus, the right-hand-side of \eqref{pequ:6} converges to zero, which leads to the contradiction and completes the proof.
\end{proof}

Our almost sure convergence result matches the ones in \cite{Na2022adaptive, Na2021Inequality} established for stochastic line search methods in constrained optimization, and matches the one in \cite{Curtis2020fully} established for stochastic trust-region method in unconstrained optimization. Compared to \cite{Curtis2020fully} (cf. Assumption 4.4 there), we do not assume the variance of the~gradient estimates decays as $\beta_k$. Such an assumption violates the flavor of fully stochastic methods, since a batch of samples is required per iteration with the batch size goes to infinity. On the contrary, we assume a growth condition (cf. Assumption \ref{ass:Ful_unbias}), which is weaker than the usual bounded variance condition. We should also mention that,~if one applies the result of \cite[Lemma 4.5]{Curtis2020fully}, one may be able to show almost sure convergence for decaying $\beta_k$ without requiring decaying variance as in the context~of~\cite{Curtis2020fully}. However, a new concern arises --- one needs to rescale the Hessian matrix~at each~step, which modifies the curvature information and affects the convergence speed.

%% file: sec4_merit.tex
\subsection{Merit parameter behavior}\label{sec:penalty}

In this subsection, we study the behavior of the merit parameter. We revisit Assumption \ref{ass:stabilized_penalty} and show that it is satisfied provided $\barg_k$ is upper bounded and $\|B_k\|$ is bounded away from zero. The condition on $\barg_k$ can be satisfied if the gradient noise has a bounded support (e.g., sampling from an empirical distribution). Such an assumption is standard to ensure a stabilized~merit parameter for both deterministic and stochastic SQP methods \citep{Bertsekas1982Constrained, Berahas2021Sequential, Berahas2021Stochastic, Berahas2022Accelerating, Curtis2021Inexact, Na2022adaptive, Na2021Inequality}.~ We~should~mention that this line of literature only assumed the existence of an upper bound on the gradient noise, which can be unknown. In other words, the bound is not involved~in~the algorithm design. In comparison, \cite{Sun2023trust} explored a bounded noise condition and incorporated the bound into the design of a trust-region algorithm. Certainly, our almost~sure convergence result also differs from the one in \cite{Sun2023trust}, which showed the iterates visited a neighborhood of stationarity infinitely often. 

Furthermore, a non-vanishing $\|B_k\|$ is a fairly mild condition, naturally \mbox{satisfied}~by all the reasonable construction methods that one uses in SQP algorithms (e.g., set $B_k$ as identity, estimated Hessian, averaged Hessian, or quasi-Newton update). However, a non-vanishing spectrum of $B_k$ is technically necessary due to our radius decomposition with the rescaled residuals (cf. \eqref{eq:breve and tilde_delta_k}). We note that a vanishing spectrum~leads~to $\breve{\Delta}_k\rightarrow0$, leading to a diminishing normal step $\bw_k$ even if we have a large feasibility residual.~The~lower~bound on $\|B_k\|$ is removable if we use original unscaled~\mbox{residuals}~to decompose the radius, or~use the alternative decomposition technique in Remark \ref{rem:3}(ii); however, an additional tuning parameter $\theta$ to balance the feasibility and optimality residuals is introduced there. We provide the analysis in Appendix \ref{appendix:A} for the sake of completeness.

\begin{assumption}\label{ass:bound_error}
For all $k\geq 0$, (i) there is a constant $M_1>0$ such that~$\|\barg_k - \nabla f_k\|\leq M_1$; and
(ii) there is a constant $\kappa_B>0$ such that $1/\kappa_B\leq\|B_k\|\leq \kappa_B$.  
\end{assumption}

\begin{lemma}\label{lemma:predk}

Suppose Assumptions \ref{ass:1-1} and \ref{ass:bound_error} hold. There exist a (potentially random) $\barK<\infty$ and a deterministic constant $\hat{\mu}$, such that $\barmu_k=\barmu_{\barK}\leq \hatmu$, $\forall k>\barK$. 
\end{lemma}

\begin{proof}
It suffices to show that there exists a deterministic threshold $\tilde{\mu}>0$ independent of $k$ such that \eqref{eq:upper_bound_predk} is satisfied as long as $\barmu_k\geq \tilde{\mu}$. We have
{\allowdisplaybreaks
\begin{align*}
\text{Pred}_k & \stackrel{\mathclap{\eqref{eq:Ful_Pred_k}}}{=}\; \barg_k^T\Delta\bx_k+\frac{1}{2}\Delta\bx_k^TB_k\Delta\bx_k+\barmu_k(\|c_k+G_k\Delta\bx_k\|-\|c_k\|)\\
& \stackrel{\mathclap{\eqref{eq:constraint_violation}}}{=} \; 
\barg_k^TZ_k\bu_k+\bargamma_k(\barg_k-\nabla f_k)^T\bv_k + \bargamma_k\nabla f_k^T\bv_k +\frac{1}{2}\bu_k^TZ_k^TB_kZ_k\bu_k+\bargamma_k\bv_k^TB_kZ_k\bu_k\\
& \quad  +\frac{1}{2}\bargamma_k^2\bv_k^TB_k\bv_k-\barmu_k\bargamma_k\|c_k\| \quad (\text{also use } \Delta\bx_k = \bargamma_k\bv_k + Z_k\bu_k)\\
& \leq (\barg_k+\bargamma_kB_k\bv_k)^TZ_k\bu_k + \frac{1}{2}\bu_k^TZ_k^TB_kZ_k\bu_k +\bargamma_k(M_1+\kappa_{\nabla f})\|\bv_k\| \\
&\quad + \frac{1}{2}\bargamma_k\|B_k\|\|\bv_k\|^2 -\barmu_k\bargamma_k\|c_k\| \quad (\text{by Assumptions \ref{ass:1-1}, \ref{ass:bound_error}}\text{ and }\bargamma_k\leq 1).
\end{align*}}
\noindent From \eqref{eq:Ful_Cauchy_2}, and replacing $\nabla \L_k$ by its stochastic estimate, we have 
\begin{align*}
\text{Pred}_k & \leq -\|\bar{\nabla}_{\bx}\L_k+\bargamma_kP_kB_k\bv_k\|\tilde{\Delta}_k+\frac{1}{2}\|B_k\|\tilde{\Delta}_k^2 +\bargamma_k(M_1+\kappa_{\nabla f})\|\bv_k\| \\
&\quad + \frac{1}{2}\bargamma_k\|B_k\|\|\bv_k\|^2 -\barmu_k\bargamma_k\|c_k\| \\
& \leq -\|\bar{\nabla}_{\bx}\L_k\|\tilde{\Delta}_k+\bargamma_k\|B_k\|\|\bv_k\|\tilde{\Delta}_k+\frac{1}{2}\|B_k\|\tilde{\Delta}_k^2 +\bargamma_k(M_1+\kappa_{\nabla f})\|\bv_k\| \\
&\quad + \frac{1}{2}\bargamma_k\|B_k\|\|\bv_k\|^2 -\barmu_k\bargamma_k\|c_k\| \quad(\text{by triangular inequality and }\|P_k\|\leq 1)\\
& \leq -\|\bar{\nabla}_{\bx}\L_k\|\Delta_k+\|\bar{\nabla}_{\bx}\L_k\|\breve{\Delta}_k+\bargamma_k\|B_k\|\|\bv_k\|\tilde{\Delta}_k+\frac{1}{2}\|B_k\|\tilde{\Delta}_k^2\\
& \quad  +\bargamma_k(M_1+\kappa_{\nabla f})\|\bv_k\| + \frac{1}{2}\bargamma_k\|B_k\|\|\bv_k\|^2 -\barmu_k\bargamma_k\|c_k\| \quad(\text{since }\tilde{\Delta}_k\geq \Delta_k-\breve{\Delta}_k)\\
& = -\|\bar{\nabla}_{\bx}\L_k\|\Delta_k-\|c_k\|\Delta_k+\|c_k\|\Delta_k+\|\bar{\nabla}_{\bx}\L_k\|\breve{\Delta}_k+\bargamma_k\|B_k\|\|\bv_k\|\tilde{\Delta}_k\\
& \quad  +\frac{1}{2}\|B_k\|\tilde{\Delta}_k^2+\bargamma_k(M_1+\kappa_{\nabla f})\|\bv_k\| + \frac{1}{2}\bargamma_k\|B_k\|\|\bv_k\|^2 -\barmu_k\bargamma_k\|c_k\|\\
& \leq -\|\bar{\nabla}\L_k\|\Delta_k+\frac{1}{2}\|B_k\|\Delta_k^2+\|c_k\|\Delta_k+\|\bar{\nabla}_{\bx}\L_k\|\breve{\Delta}_k+\bargamma_k\|B_k\|\|\bv_k\|\Delta_k\\
& \quad  +\bargamma_k(M_1+\kappa_{\nabla f})\|\bv_k\| + \frac{1}{2}\bargamma_k\|B_k\|\|\bv_k\|^2 -\barmu_k\bargamma_k\|c_k\|,
\end{align*}
\noindent since $\|\bar{\nabla}_{\bx}\L_k\|+\|c_k\|\geq \|\bar{\nabla}\L_k\|$ and $\tilde{\Delta}_k\leq\Delta_k$. 
Thus, \eqref{eq:upper_bound_predk} holds as long as 
\begin{equation*}
 \barmu_k\bargamma_k\|c_k\|\geq    \|c_k\|\Delta_k+\|\bar{\nabla}_{\bx}\L_k\|\breve{\Delta}_k+\bargamma_k\|B_k\|\|\bv_k\|\Delta_k +\bargamma_k(M_1+\kappa_{\nabla f})\|\bv_k\| + \frac{\bargamma_k}{2}\|B_k\|\|\bv_k\|^2.
\end{equation*}
\noindent Since $\|\bv_k\|\leq \|c_k\|/\sqrt{\kappa_{1,G}}$ and $ \Delta_k\leq \Delta_{\max}\coloneqq\eta_{\max}\alpha_u\beta_{\max}(\kappa_c+M_1+\kappa_{\nabla f})$ (cf.~Assumption \ref{ass:1-1} and Lemma \ref{lemma:4statements}), it is sufficient to show
\begin{equation}\label{eq:mu_1}
\barmu_k\bargamma_k\|c_k\|\geq    \|c_k\|\Delta_k+\|\bar{\nabla}_{\bx}\L_k\|\breve{\Delta}_k+\bargamma_k\|c_k\|\left(\frac{\kappa_B\Delta_{\max}+M_1+\kappa_{\nabla f}}{\sqrt{\kappa_{1,G}}}+\frac{\kappa_B\kappa_c}{2\kappa_{1,G}}\right).
\end{equation}
Equivalently, 
\begin{equation*}
\barmu_k\geq    \frac{\Delta_k}{\bargamma_k}+\frac{\|\bar{\nabla}_{\bx}\L_k\|\breve{\Delta}_k}{\bargamma_k\|c_k\|}+\left(\frac{\kappa_B\Delta_{\max}+M_1+\kappa_{\nabla f}}{\sqrt{\kappa_{1,G}}}+\frac{\kappa_B\kappa_c}{2\kappa_{1,G}}\right).
\end{equation*}
\noindent We only consider $\|c_k\|> 0$ since \eqref{eq:mu_1} holds when $\|c_k\|=0$. By \eqref{Ful_GenerateRadius}, we find that
\begin{equation*}
\frac{\Delta_k}{\bargamma_k}+\frac{\|\bar{\nabla}_{\bx}\L_k\|\breve{\Delta}_k}{\bargamma_k\|c_k\|}\leq \frac{\eta_{1,k}\alpha_k\|\bar{\nabla}\L_k\|}{\bargamma_k}\left(1+\frac{\|\bar{\nabla}_{\bx}\L_k\|\|G_k\|^{-1}}{\|\bar{\nabla}\L_k^{RS}\|}\right).
\end{equation*}
Noticing that $\|\bar{\nabla}\L_k^{RS}\|\geq \min\{\|B_k\|^{-1},\|G_k\|^{-1}\}\|\bar{\nabla}\L_k\|$, we find 
\begin{align*}
\frac{\Delta_k}{\bargamma_k}+\frac{\|\bar{\nabla}_{\bx}\L_k\|\breve{\Delta}_k}{\bargamma_k\|c_k\|} & \leq \frac{\eta_{1,k}\alpha_k\|\bar{\nabla}\L_k\|}{\bargamma_k}\left(1+\frac{\|\bar{\nabla}_{\bx}\L_k\|\|G_k\|^{-1}}{\min\{\|B_k\|^{-1},\|G_k\|^{-1}\}\|\bar{\nabla}\L_k\|}\right)\\
& = \frac{\eta_{1,k}\alpha_k\|\bar{\nabla}\L_k\|}{\bargamma_k}\left(1+\max\left\{\frac{\|B_k\|}{\|G_k\|},1\right\}\frac{\|\bar{\nabla}_{\bx}\L_k\|}{\|\bar{\nabla}\L_k\|}\right)\\
& \leq \frac{2\eta_{1,k}\alpha_k\|\bar{\nabla}\L_k\|}{\bargamma_k}\max\left\{\frac{\|B_k\|}{\|G_k\|},1\right\}.
\end{align*}
To analyze $\bargamma_k$, we notice that $\|\bar{\nabla}\L_k^{RS}\|\leq \max\{\|B_k\|^{-1},\|G_k\|^{-1}\}\|\bar{\nabla}\L_k\|$. Therefore, 
\begin{multline*}
\frac{\breve{\Delta}_k}{\|\bv_k\|} \stackrel{\eqref{eq:breve and tilde_delta_k}}{=} \frac{\|c_k^{RS}\|\Delta_k}{\|\bar{\nabla}\L_k^{RS}\|\|\bv_k\|} \stackrel{\eqref{Ful_GenerateRadius}}{\geq} \frac{\eta_{2,k}\alpha_k\|G_k\|^{-1}\|c_k\|}{\max\{\|B_k\|^{-1},\|G_k\|^{-1}\}\|\bv_k\|} \\
\geq \frac{\eta_{1,k}\alpha_k\|c_k\|}{2\|\bv_k\|}\min\left\{\frac{\|B_k\|}{\|G_k\|},1\right\}\stackrel{\eqref{def:eta2k}}{=}\zeta\alpha_k\phi_k/2,
\end{multline*}
where the last inequality is due to the fact that \eqref{def:eta2k} implies $\zeta\alpha_k\leq 1$, implying~$\eta_{2,k}\geq \eta_{1,k}/2$. We therefore have
\begin{equation}\label{equ:4.13}
\frac{1}{2}\zeta\alpha_k\phi_k\leq \min\left\{\breve{\Delta}_k/\|\bv_k\|, 1\right\}\eqqcolon \bar{\gamma}_k^{\text{trial}}.
\end{equation}
\noindent The above display suggests that we only need to consider $\bargamma_k=\frac{1}{2}\zeta\alpha_k\phi_k$. Noting that $\max\{\|B_k\|/\|G_k\|,1\}\leq \max\{\kappa_B/\sqrt{\kappa_{1,G}},1\}$, $\min\{\|B_k\|/\|G_k\|,1\}\geq \min\{1/(\kappa_B\sqrt{\kappa_{2,G}}),1\}$ and $\|\bar{\nabla}\L_k\|\leq \kappa_c+M_1+\kappa_{\nabla f}$, we obtain that 
\begin{equation*}
\frac{\Delta_k}{\bargamma_k}+\frac{\|\bar{\nabla}_{\bx}\L_k\|\breve{\Delta}_k}{\bargamma_k\|c_k\|}\leq \left[\frac{4\eta_{\max}}{\zeta}(\kappa_c+\kappa_{\nabla f}+M_1)\max\left\{\frac{\kappa_B}{\sqrt{\kappa_{1,G}}},1\right\}\right]\cdot \max\{\kappa_B\sqrt{\kappa_{2,G}},1\}.
\end{equation*}
Therefore, \eqref{eq:upper_bound_predk} holds as long as 
\begin{multline*}
\barmu_k\geq \tilde{\mu}\coloneqq\left[\frac{4\eta_{\max}}{\zeta}(\kappa_c+\kappa_{\nabla f}+M_1)\max\left\{\frac{\kappa_B}{\sqrt{\kappa_{1,G}}},1\right\}\right]\cdot \max\{\kappa_B\sqrt{\kappa_{2,G}},1\}\\
+\left(\frac{\kappa_B\Delta_{\max}+M_1+\kappa_{\nabla f}}{\sqrt{\kappa_{1,G}}}+\frac{\kappa_B\kappa_c}{2\kappa_{1,G}}\right).
\end{multline*}
\noindent Since $\barmu_k$ is increased by at least a factor of $\rho$ for each update, we define $\hat{\mu}\coloneqq\rho\tilde{\mu}$ and complete the proof.
\end{proof}

Compared to existing StoSQP methods, we do not require the stabilized merit~parameter to be large enough. The additional requirement of having a large enough~stabilized value is critical for existing StoSQP methods.~To satisfy this requirement,~\cite{Na2022adaptive, Na2021Inequality}~imposed an adaptive condition on the feasibility error to be satisfied when selecting the merit parameter; and \cite{Berahas2021Sequential,Berahas2021Stochastic,Berahas2022Accelerating,Curtis2021Inexact} imposed a symmetry condition on the noise~distribution. Intuitively, the reduction of the merit function in StoSQP methods should be related to the true KKT residual. In the aforementioned methods, the reduction~of~the stochastic merit function model is first related to the reduction of the deterministic merit function model, and then related to the true KKT residual. However, the~relation between the reduction in stochastic and deterministic models is only valid when the merit parameter stabilizes at a sufficiently large value \cite[Lemma 3.12]{Berahas2021Sequential}. In contrast, our approach relates the reduction of stochastic model to the squared \textit{estimated} KKT residual $\|\bar{\nabla}\L_k\|^2$ (i.e., \eqref{eq:upper_bound_predk}). After \mbox{taking}~the~\mbox{conditional}~\mbox{expectation}~and~\mbox{carefully}~analyzing the error terms, we can further use the true KKT residual~to characterize the improvement of the merit function in each step. In the end, we suppress the condition on a sufficiently large~merit~parameter.

%% file: sec5.tex
\section{Numerical Experiments}\label{sec:5}

We demonstrate the empirical performance of~Algorithm \ref{Alg:Non-adap} and compare it to the line-search $\ell_1$-StoSQP method designed in \citep[Algorithm 3]{Berahas2021Sequential} under the same fully stochastic setup. We describe the algorithmic \mbox{settings}~in Section~\ref{sec:5.1}; then we show numerical results on a subset of CUTEst problems \citep{Gould2014CUTEst}~in Section~\ref{sec:5.2}; and then we show numerical results on constrained logistic~regression problems in Section \ref{sec:5.3}. The implementation of TR-StoSQP is available at \url{https://github.com/ychenfang/TR-StoSQP}.

\subsection{Algorithm setups}
\label{sec:5.1}

For both our method and $\ell_1$-StoSQP, we try two constant sequences, $\beta_k\in\{0.5,1\}$, and two decaying sequences, $\beta_k\in\{k^{-0.6},k^{-0.8}\}$.~The sequence $\{\beta_k\}$ is used to select the stepsize in $\ell_1$-StoSQP. We use the same input~since, as discussed in Remark \ref{rem:1}, $\beta_k$ in two methods shares the same order. For both methods, the Lipschitz constants of the objective gradients and constraint Jacobians are estimated around the initialization and kept constant for subsequent iterations.

We follow \citet{Berahas2021Sequential} to set up the $\ell_1$-StoSQP method, where we set $B_k=I$ and solve~the SQP subproblems exactly. We set the parameters of~TR-StoSQP as $\zeta=10$, $\delta=10$, $\barmu_{-1}=1$, and $\rho=1.5$. We use \texttt{IPOPT} solver \citep{Waechter2005implementation} to solve \eqref{eq:Sto_tangential_step}, and apply four different Hessian approximations $B_k$ as follows:
\begin{enumerate}[label=(\alph*),topsep=0pt]
\setlength\itemsep{0.2em}
\item Identity (Id). We set $B_k = I$, which is widely used in the literature \citep{Berahas2021Stochastic,Berahas2021Sequential,Na2021Inequality,Na2022adaptive}.

\item Symmetric rank-one (SR1) update.
We set $H_{-1}=H_{0}=I$ and update $H_k$ as
\begin{equation*}
H_{k}=H_{k-1}+\frac{(\boldsymbol{y}_{k-1}-H_{k-1}\Delta\bx_{k-1})(\boldsymbol{y}_{k-1}-H_{k-1}\Delta\bx_{k-1})^T}{(\boldsymbol{y}_{k-1}-H_{k-1}\Delta\bx_{k-1})^T\Delta\bx_{k-1}}, \quad \forall k\geq 1,
\end{equation*}
\noindent where $\boldsymbol{y}_{k-1}=\bar{\nabla}_{\bx}\L_{k}-\bar{\nabla}_{\bx}\L_{k-1}$ and $\Delta\bx_{k-1}=\bx_{k}-\bx_{k-1}$. Since $H_{k}$ depends~on $\barg_k$, we set $B_k = H_{k-1}$ ($B_0=H_{-1}=I$) to ensure that $\sigma(B_k)\subseteq\mF_{k-1}$.

\item Estimated Hessian (EstH). We set $B_0=I$ and $B_k=\bar{\nabla}_{\bx}^2\L_{k-1}$, $\forall k\geq 1$,~where $\bar{\nabla}_{\bx}^2\L_{k-1}$ is estimated using the same sample used to estimate $\barg_{k-1}$.

\item Averaged Hessian (AveH). We set $B_0=I$, set $B_k= \sum_{i=k-100}^{k-1}\bar{\nabla}_{\bx}^2\L_i/100$ for $k\geq 100$, and set $B_k = \sum_{i=0}^{k-1}\bar{\nabla}_{\bx}^2\L_i/k$ for $0<k<100$.
This Hessian~approximation is inspired by \citet{Na2022Hessian}, where the authors showed that the Hessian averaging is helpful for denoising the noise in the stochastic Hessian estimates.
\end{enumerate}

\subsection{CUTEst}\label{sec:5.2}

We select problems from the CUTEst set that have a non-constant objective with only equality constraints, satisfy $d < 1000$, and do not report singularity on $G_kG_k^T$ during the iteration process, resulting in 47 problems in total. The~initial iterate is provided by the CUTEst package. At each step, the \mbox{estimate}~$\barg_k$~is~drawn~from $\N(\nabla f_k, \sigma^2(I+\boldsymbol{1}\boldsymbol{1}^T))$, where $\boldsymbol{1}$ denotes the $d$-dimensional all one vector and $\sigma^2$ denotes the noise level varying within $\{ 10^{-8},10^{-4},10^{-2},10^{-1}\}$. When the approximation EstH or AveH is used, the estimate $(\bar{\nabla}^2 f_k)_{i,j}$ (same for the $(j,i)$ entry) is drawn from $\N((\nabla^2 f_k)_{i,j}, \sigma^2)$ with the same $\sigma^2$ used for estimating the gradient. We set the iteration~budget to $10^5$ and, for each setup of $\beta_k$ and $\sigma^2$, average the KKT residuals over 5 runs. We stop the iteration of both methods if $\|\nabla\mathcal{L}_k\|\leq 10^{-4}$ or $k\geq 10^5$.

We report the KKT residuals of $\ell_1$-StoSQP and TR-StoSQP with different Hessian approximations in Figure \ref{fig:TR-LScompare}. We observe that for both constant $\beta_k$ and decaying $\beta_k$ with a high noise level, \alg\ consistently outperforms $\ell_1$-StoSQP. We note~that $\ell_1$-StoSQP performs better than \alg\ for decaying $\beta_k$ with a low noise level (e.g., $\sigma^2=10^{-8}$). However, in that case, \alg\ is not sensitive to the noise~level $\sigma^2$ while the performance of $\ell_1$-StoSQP deteriorates rapidly as $\sigma^2$ increases. We think that the robustness against noise is a benefit brought by the trust-region constraint, which properly regularizes the SQP subproblem when $\sigma^2$ is large. Furthermore, among the four choices of Hessian approximations, \alg\ generally performs the best with the averaged Hessian, and the second best with the estimated Hessian. Compared to the identity and SR1 update, the estimated~\mbox{Hessian provides a better} approximation to the true Hessian (especially when $\sigma^2$ is small); the averaged Hessian further reduces the noise that leads to a better performance (especially when $\sigma^2$ is large). 

We observe that when $\sigma^2$ is large, or $\sigma^2$ is small but $\beta_k$ is constant, TR-StoSQP outperforms $\ell_1$-StoSQP even when the identity Hessian is used. However, for decaying $\beta_k$ and small $\sigma^2$, the performance of TR-StoSQP is less competitive. This disparity~in performance could arise from the difference in trial step computation.~In line-search methods, even though the search direction is~\mbox{determined}~by~solving a Newton system, it can still be decomposed orthogonally into a normal direction $\bw_k\in\text{im}(G_k^T)$ and a tangential direction $\bt_k\in\text{ker}(G_k)$ \citep[see][for details]{Berahas2021Sequential}. The~direction of $\bw_k$ is consistent between trust-region and line-search methods, represented~as 
$\bv_k\coloneqq -G_k^T[G_kG_k^T]^{-1}c_k$. However, the directions of the tangential step are \mbox{different}.~In trust-region methods, the tangential step is determined by \eqref{eq:Sto_tangential_step} \mbox{using}~$\bw_k=\bargamma_k\bv_k$~with $\bargamma_k$ chosen based on \eqref{eq:Sto_gamma_k} and \eqref{def:proj}. In contrast, in line-search methods, the tangential direction effectively comes from \eqref{eq:Sto_tangential_step} using $\bw_k=\bv_k$ without the trust-region~constraint. In stochastic optimization, most iterations satisfy $\bargamma_k<1$. Therefore, the~directions of the tangential step might differ in trust-region methods and~\mbox{line-search}~\mbox{methods},~even~if~the~identity
Hessians are used and the iterates~are near an optimal point. Also, the \mbox{trust-region}~constraint serves as a regularization that is potentially helpful for large noise scenarios~($\sigma^2$ is large or $\beta_k$ is constant).
We should emphasize that the difference in trial step direction is due to different mechanisms of trust-region methods and line-search methods (trust-region methods compute the search direction and stepsize simultaneously, while line-search methods compute them separately) and the fully stochastic setup~(the~noise does not gradually vanish), but not due to our algorithm design.

\begin{figure}[!t]
\centering
\subfigure[$\beta_k=0.5$]{\includegraphics[width=0.43\textwidth]{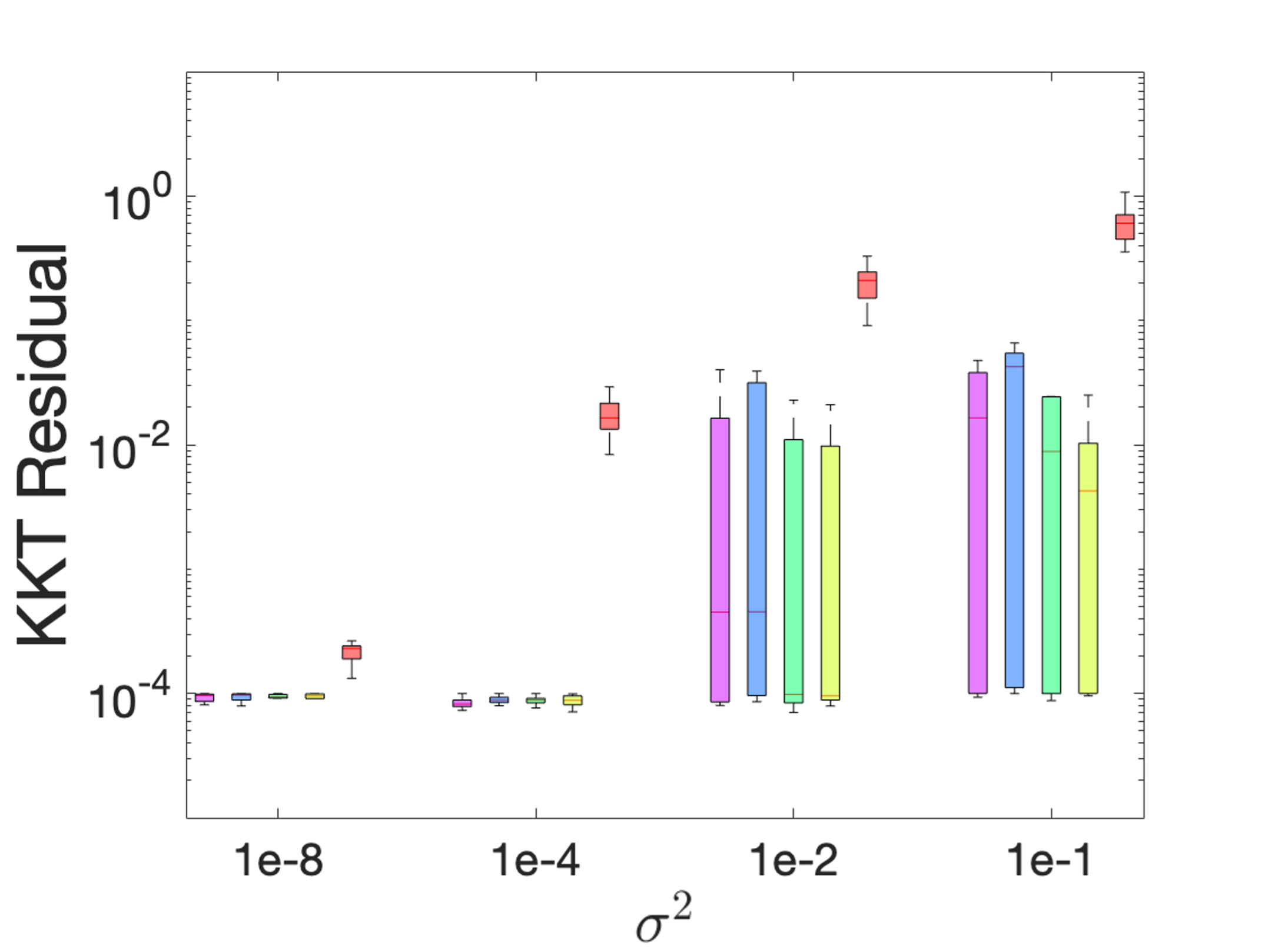}}
\subfigure[$\beta_k=1.0$]{\includegraphics[width=0.43\textwidth]{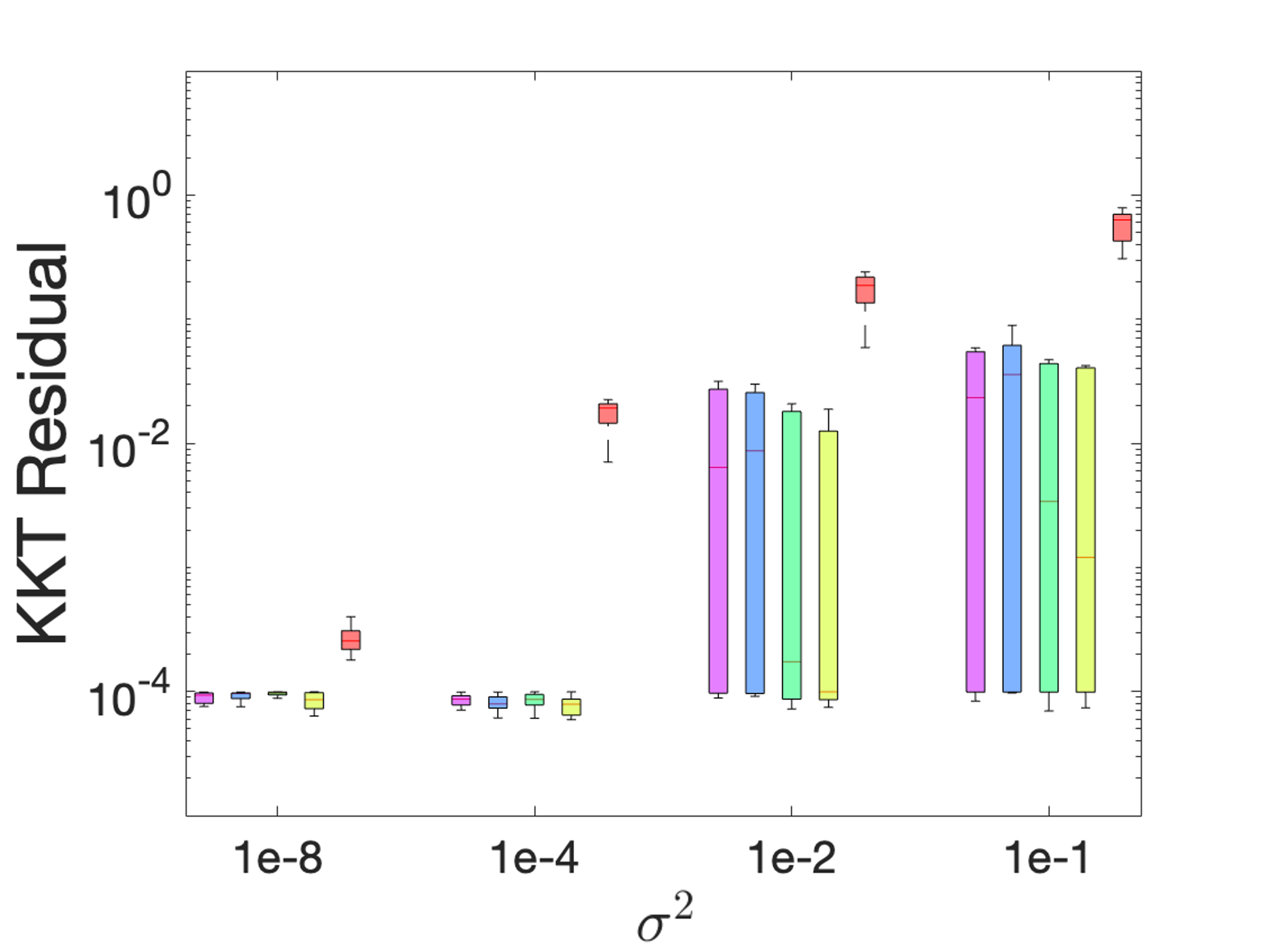}}
\subfigure[$\beta_k=k^{-0.6}$]{\includegraphics[width=0.43\textwidth]{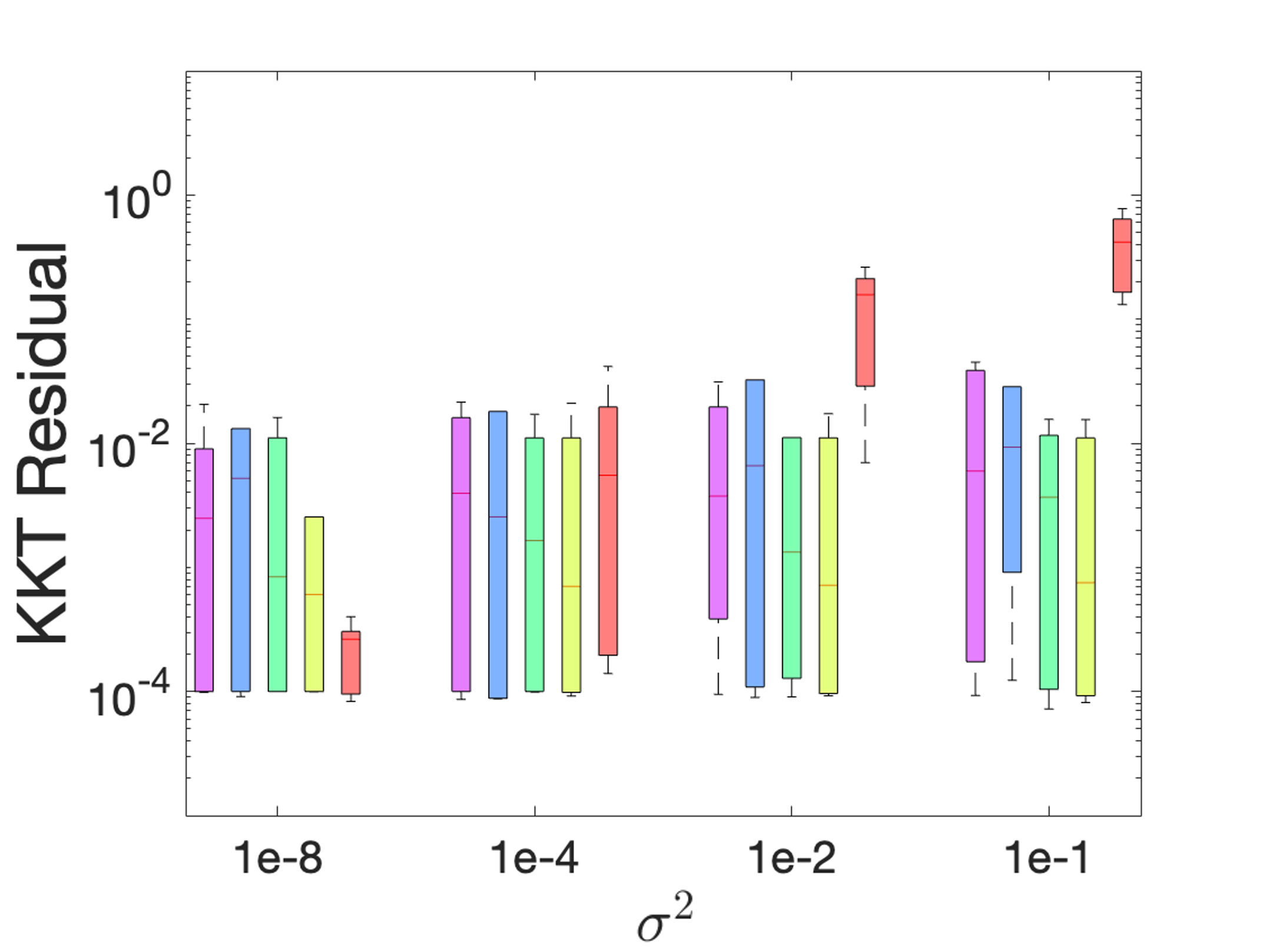}}
\subfigure[$\beta_k=k^{-0.8}$]{\includegraphics[width=0.43\textwidth]{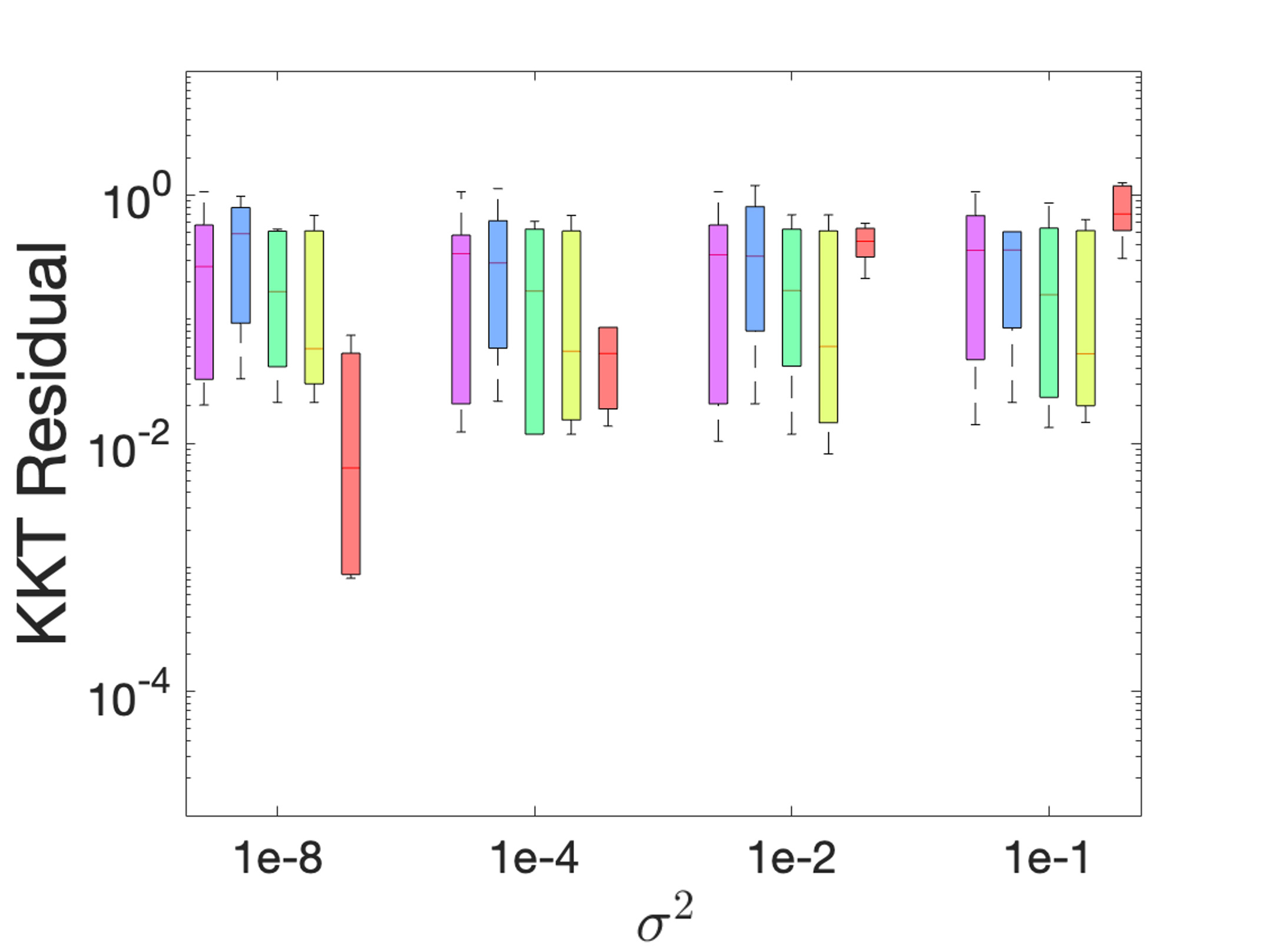}}
\subfigure{\includegraphics[width=0.7\textwidth]{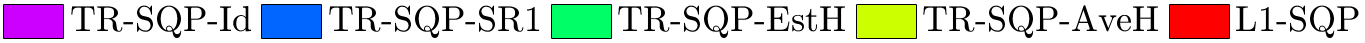}}
\caption{KKT residual boxplots for CUTEst problems. For each $\sigma^2$, there are five boxes. The first four boxes correspond to the proposed \alg\ method with four different choices of $B_k$, while the last box corresponds to the $\ell_1$-StoSQP method.}\label{fig:TR-LScompare}
\end{figure}

We then investigate the adaptivity of the radius selection scheme in \eqref{Ful_GenerateRadius}. As~explained in Remark \ref{rem:1}, the radius $\Delta_k$ can be set larger or smaller than $\alpha_k = \mathcal{O}(\beta_k)$, depending on the magnitude of the estimated KKT residual. In Table \ref{tab:portion}, we report~the proportions of the three cases in \eqref{Ful_GenerateRadius}: $\Delta_k<\alpha_k$, $\Delta_k=\alpha_k$, and $\Delta_k>\alpha_k$.~We \mbox{average} the proportions over 5 runs of all 47 problems in each setup. From Table \ref{tab:portion}, we have~the following three observations. (i) Case 2 has a near zero proportion for all setups. This phenomenon is due to the fact that $\eta_{1,k}-\eta_{2,k}=\mathcal{O}(\beta_k)$. For constant~$\beta_k$,~this value~is small, thus a few iterations are in Case 2.~For decaying $\beta_k$, this value even converges~to zero, thus almost no iterations are in Case 2. (ii) Case 3 is triggered~quite frequently if $\beta_k$ decays rapidly. This phenomenon suggests that the adaptive scheme can generate aggressive steps even if we input a conservative~radius-related sequence $\beta_k$. (iii)~The proportion of Case 1 dominates the other two cases in the most of setups. This is reasonable since Case 1 is always triggered when the iterates are near a KKT point.

\begin{table}[t]
\centering
\addtolength{\tabcolsep}{-0.15cm} 
\resizebox{\textwidth}{!}{
\begin{tabular}{cc|ccc|ccc|ccc|ccc}
\hline
\multicolumn{1}{c}{\multirow{2}{*}{$\beta_k$}} & \multicolumn{1}{c|}{\multirow{2}{*}{$B_k$}}  & \multicolumn{3}{c|}{$\sigma^2=10^{-8}$}              & \multicolumn{3}{c|}{$\sigma^2=10^{-4}$}              & \multicolumn{3}{c|}{$\sigma^2=10^{-2}$}              & \multicolumn{3}{c}{$\sigma^2=10^{-1}$}             
 \\ \cline{3-14} 
&  & \multicolumn{1}{c}{\small Case 1} & \multicolumn{1}{c}{\small Case 2} & \multicolumn{1}{c|}{\small Case 3} & \multicolumn{1}{c}{\small Case 1} & \multicolumn{1}{c}{\small Case 2} & \multicolumn{1}{c|}{\small Case 3} & \multicolumn{1}{c}{\small Case 1} & \multicolumn{1}{c}{\small Case 2} & \multicolumn{1}{c|}{\small Case 3} & \multicolumn{1}{c}{\small Case 1} & \multicolumn{1}{c}{\small Case 2} & \multicolumn{1}{c}{\small Case 3} \\ \hline
\multirow{4}{*}{0.5} & Id & 90.3 & 0.1 & 9.6 & 91.3 & 0.2 & 8.5 & 95.0 & 0.1 & 4.9 & 54.7 & 0.9 & \textbf{44.4} \\
& SR1 & 93.8 & 0.1 & 6.1 & 92.7 & 0.1 & 7.2 & 94.6 & 0.1 & 5.7 & 56.2 & 1.1 & \textbf{42.7} \\
& EstH	& 92.2 & 0.1 & 7.7 & 94.8 & 0.1 & 5.1 & 84.8 & 0.2 & 15.0 & 71.1 & 0.5 & \textbf{28.4} \\
& AveH & 92.5 & 0.1 & 7.4 & 94.1 & 0.1 & 5.8 & 88.2 & 0.2 & 11.6 & 64.2 & 0.4 & \textbf{35.4} \\ \hline
\multirow{4}{*}{1.0} & Id & 92.0 & 0.1 & 7.9 & 93.7 & 0.1 & 6.2 & 95.4 & 0.2 & 4.4 & 57.1 & 1.2 & \textbf{41.7}\\
& SR1 & 94.0 & 0.2 & 5.8 & 96.1 & 0.1 & 3.8 & 97.7 & 0.2 & 2.1 & 64.2 & 1.2 & \textbf{34.6} \\
& EstH & 92.4 & 0.1 & 7.5 & 93.8 & 0.1 & 6.1 & 87.5  & 0.4 & 12.1 & 72.8 & 0.5 & \textbf{26.7} \\
& AveH & 92.4 & 0.2 & 7.4 & 93.9 & 0.3 & 5.8 & 85.5 & 0.3 & 14.2 & 67.1 & 0.6 & \textbf{32.3} \\ \hline
\multirow{4}{*}{$k^{-0.6}$} & Id & 97.2 & 0.0 & 2.8 & 96.8 & 0.0 & 3.2 & 93.4 & 0.0 & 6.6 & 51.8 & 0.0 & \textbf{48.2} \\
& SR1 & 98.3 & 0.0 & 1.7 & 97.1 & 0.0 & 2.9 & 93.2 & 0.0 & 6.8 & 51.5 & 0.0 & \textbf{48.5} \\
& EstH	& 97.9 & 0.0 & 2.1 & 95.8 & 0.0 & 4.2 & 86.6 & 0.0 & 13.4 & 69.1 & 0.0 & \textbf{30.9} \\
& AveH & 97.4 & 0.0 & 2.6 & 96.1 & 0.0 & 3.9 & 86.8 & 0.0 & 13.2 & 65.5 & 0.0 & \textbf{34.8} \\ \hline
\multirow{4}{*}{$k^{-0.8}$} & Id & 70.6 & 0.0 & \textbf{29.4} & 68.1 & 0.0 & \textbf{31.9} & 66.4 & 0.0 & \textbf{33.6} & 45.8 & 0.0 & \textbf{54.2} \\
& SR1 & 56.1 & 0.0 & \textbf{43.9} & 65.7 & 0.0 & \textbf{34.3} & 66.6 & 0.0 & \textbf{33.4} & 39.9 & 0.0 & \textbf{60.1} \\
& EstH & 67.5 & 0.0 & \textbf{32.5} & 65.2 & 0.0 & \textbf{34.8} & 62.0 & 0.0 & \textbf{38.0} & 54.7 & 0.0 & \textbf{45.3} \\
& AveH & 67.9 & 0.0 & \textbf{32.1} & 66.7 & 0.0 & \textbf{33.3} & 65.9 & 0.0 & \textbf{34.1} & 51.4 & 0.0 & \textbf{48.6} \\
\hline     
\end{tabular}}
\caption{Proportions of the three cases in \eqref{Ful_GenerateRadius} (\%). We highlight the proportion of Case~3~if the value is higher than 25\%.}\label{tab:portion}
\end{table}

In Remark \ref{rem:3}, we provide two alternative relaxation techniques to compute~the trial step. Figure \ref{fig:Methodcompare} reports the KKT residuals for these methods. We use \texttt{Adap1}~to~denote TR-StoSQP with our adaptive relaxation technique; \texttt{Adap2} to denote TR-StoSQP with the technique in Remark \ref{rem:3}(i), where the radius of the tangential step is controlled by $\tilde{\Delta}_k\coloneqq\sqrt{\Delta_k^2-\|\bw_k\|^2}$; and \texttt{NonAdap} to denote TR-StoSQP with the technique in Remark \ref{rem:3}(ii), where the prespecified parameter is set as $\theta=0.8$. The~remaining algorithm setups follow from TR-StoSQP and $B_k=I$. We observe that~the three techniques have comparable performance for most \mbox{combinations}~of~$\beta_k$~and~$\sigma^2$, while \texttt{Adap1} is slightly better than the other two techniques in some cases. 	
The results suggest that our adaptive relaxation technique, as well as its variant in Remark~\ref{rem:3}(i), is at least as good as the conventional technique (the nonadaptive technique in Remark \ref{rem:3}(ii)) in practice, but it requires no effort in tuning parameters.

\begin{figure}[!thb]
\centering
\subfigure[$\beta_k=0.5$]{\includegraphics[width=0.43\textwidth]{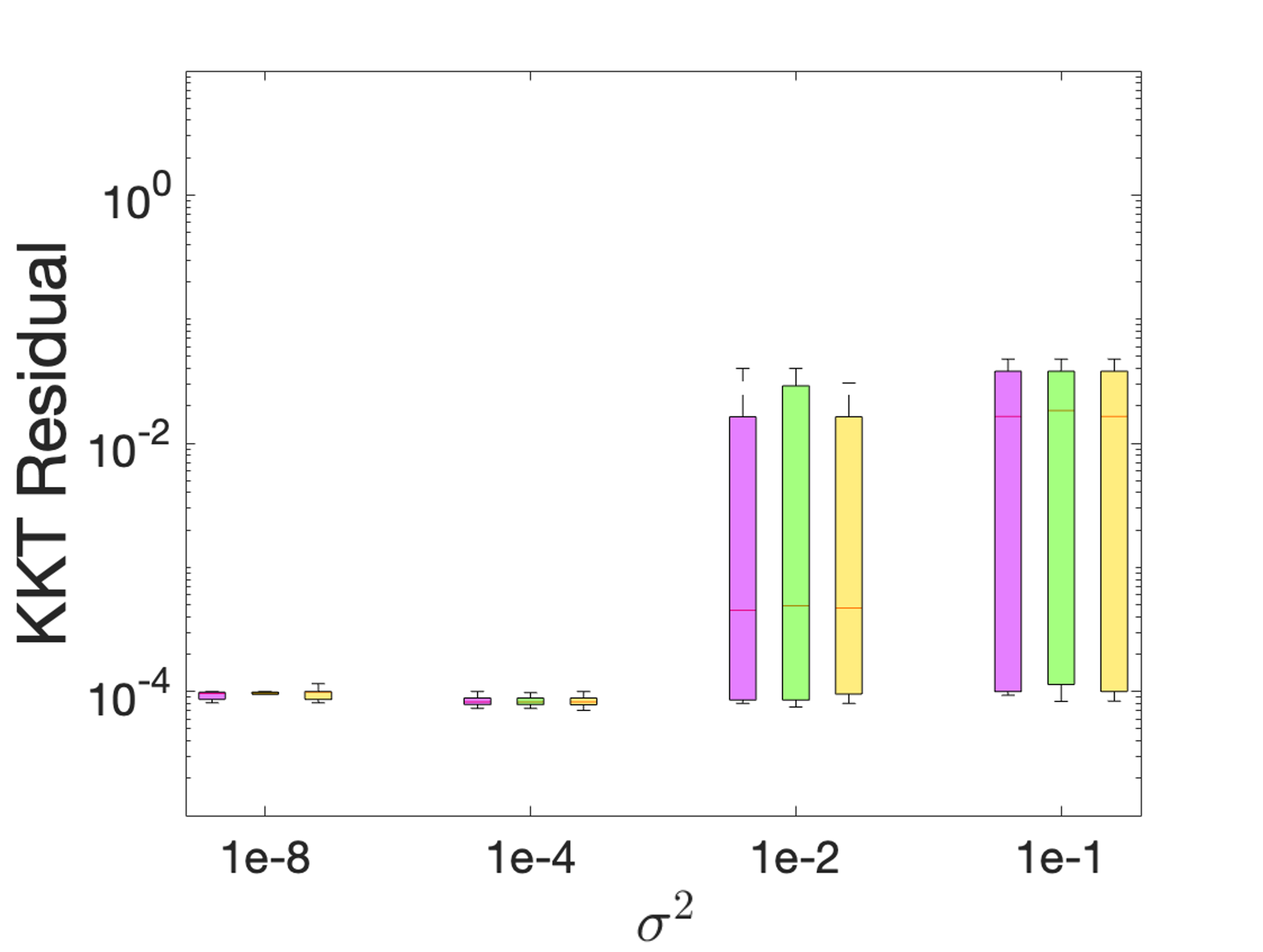}}
\subfigure[$\beta_k=1.0$]{\includegraphics[width=0.43\textwidth]{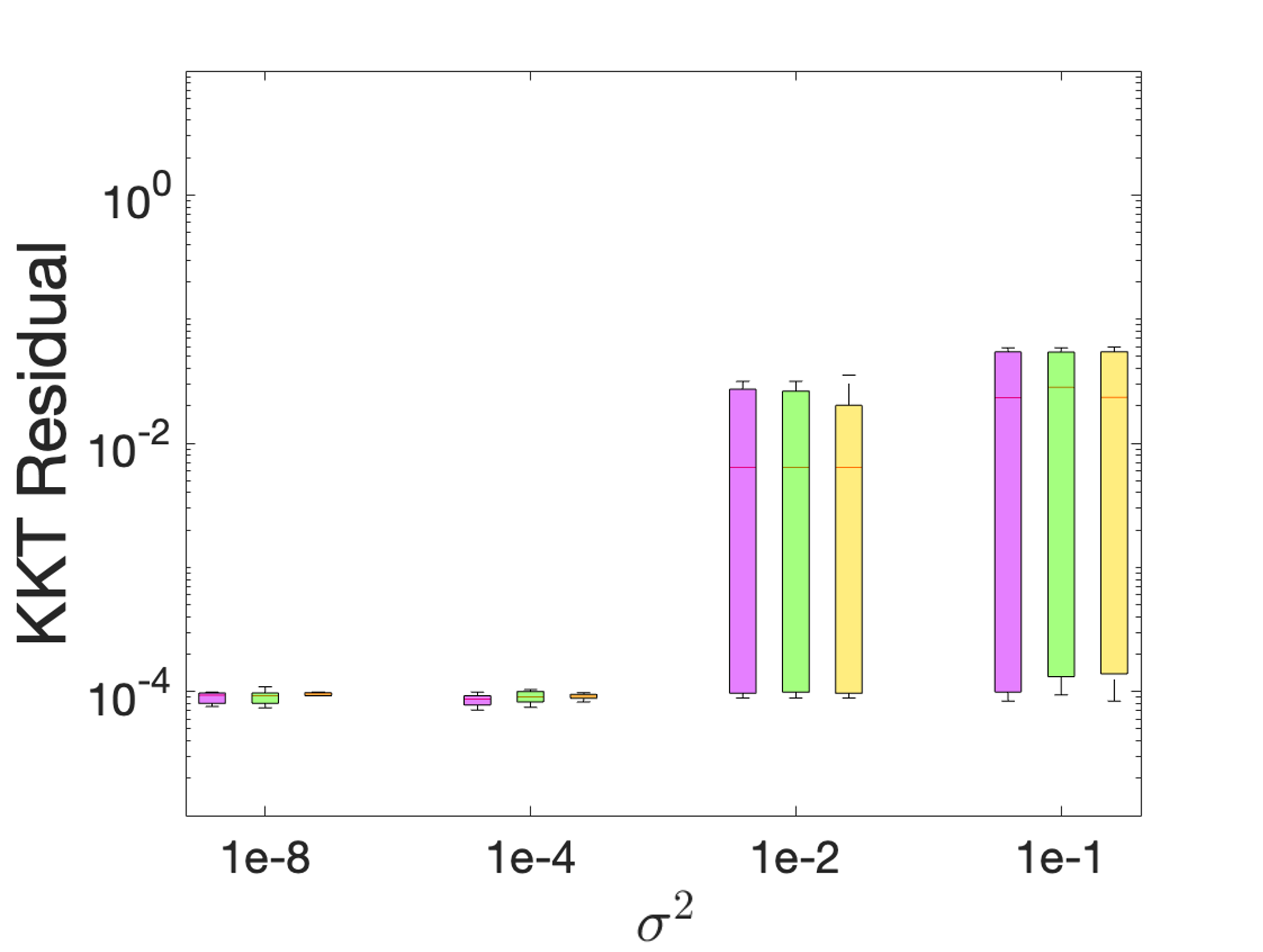}}
\subfigure[$\beta_k=k^{-0.6}$]{\includegraphics[width=0.43\textwidth]{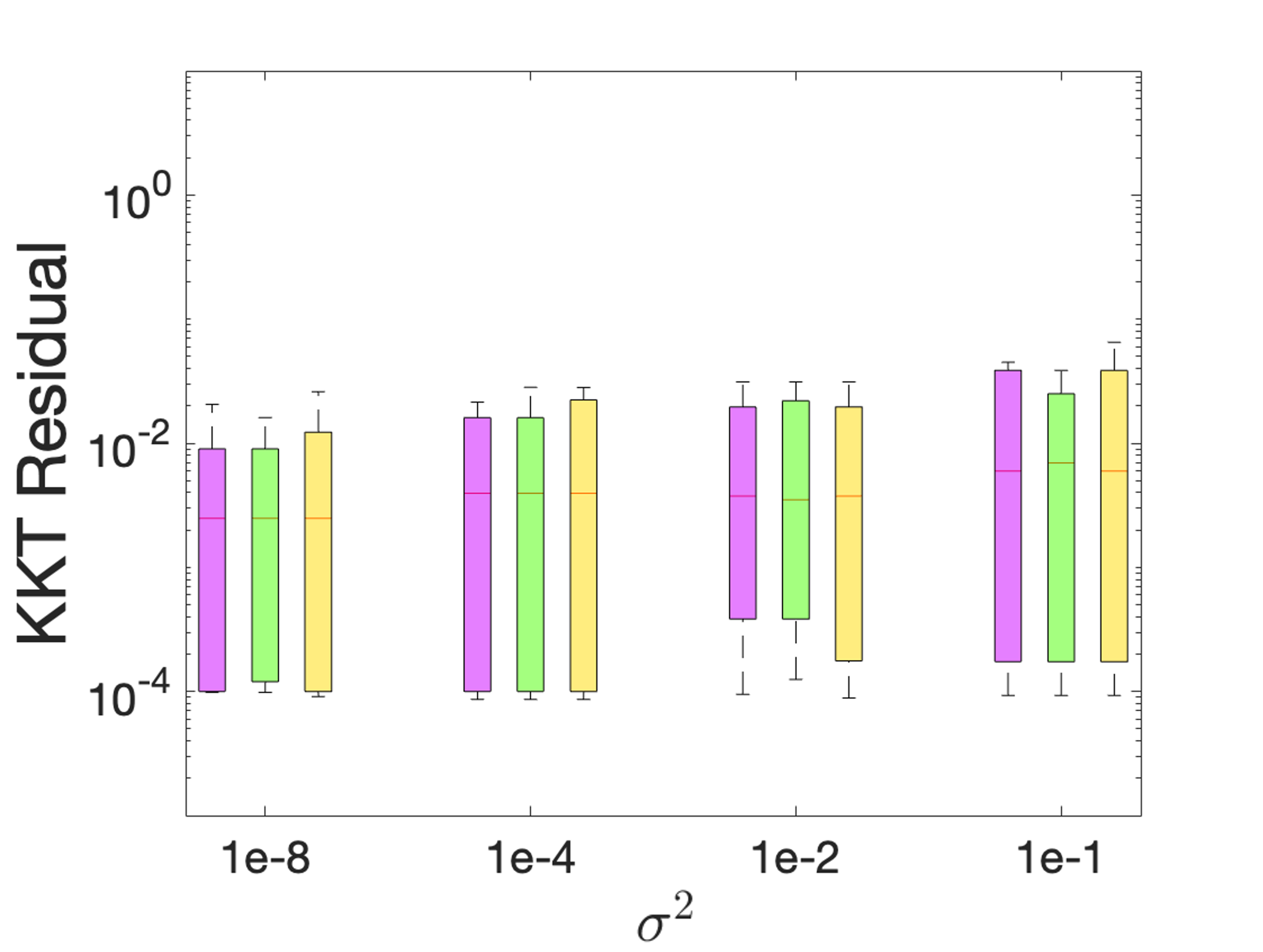}}
\subfigure[$\beta_k=k^{-0.8}$]{\includegraphics[width=0.43\textwidth]{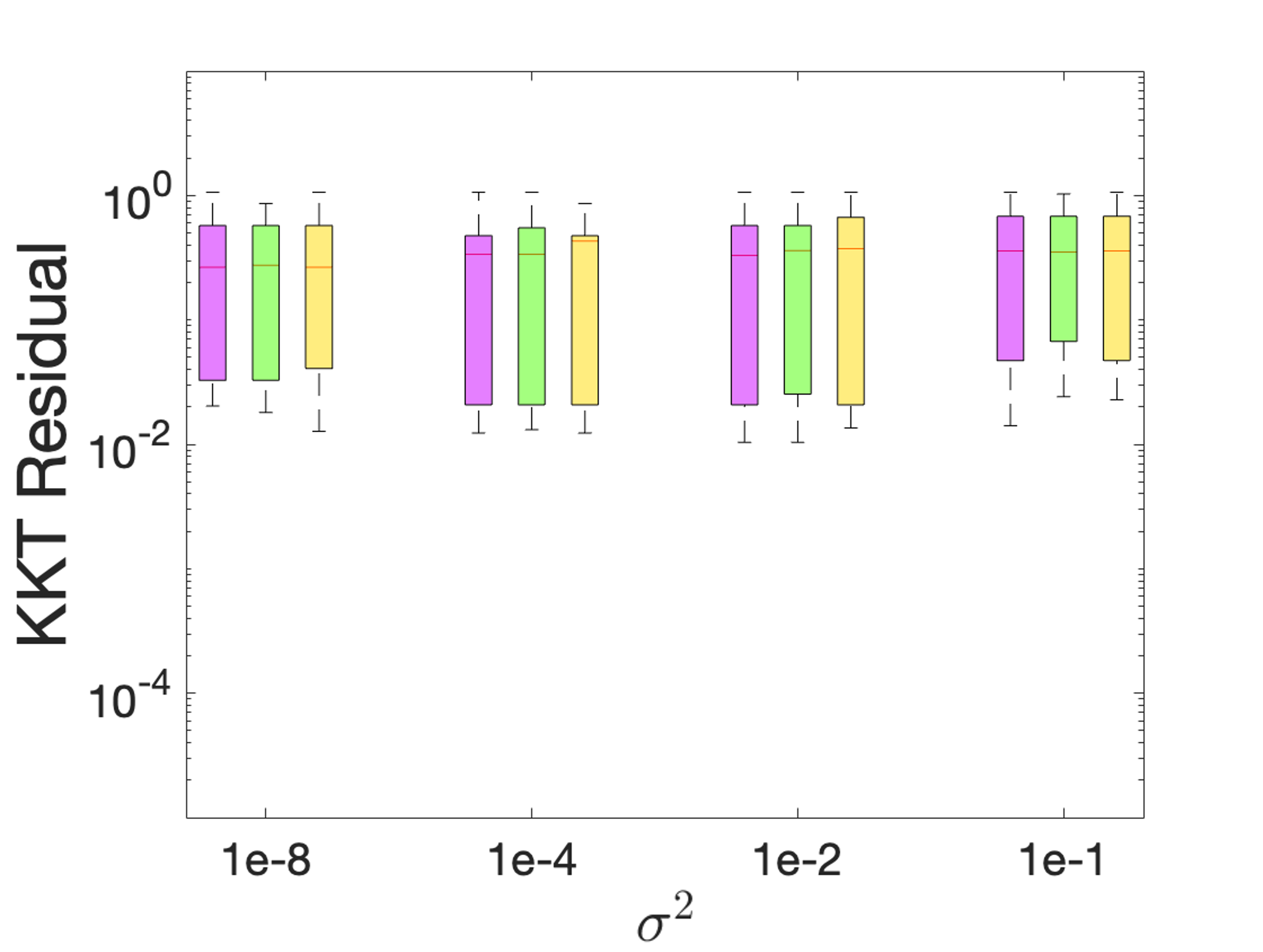}}
\subfigure{\includegraphics[width=0.35\textwidth]{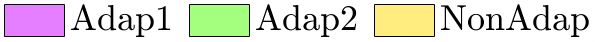}}
\caption{KKT residual boxplots for CUTEst problems with different relaxation techniques. The Hessian approximation $B_k$ is set as identity matrix. For each $\sigma^2$, there are three boxes. The first box corresponds to the proposed adaptive relaxation technique. The second box corresponds to the adaptive technique in Remark \ref{rem:3} (i). The last box corresponds to the nonadaptive technique in Remark \ref{rem:3} (ii).}\label{fig:Methodcompare}

\end{figure}

\subsection{Constrained logistic regression}\label{sec:5.3}

We consider equality-constrained logistic regression of the form 
\begin{equation*}
\min_{\bx\in\mR^d}\; f(\bx)=\frac{1}{N}\sum_{i=1}^N\log\cbr{1+\exp\rbr{-y_i\cdot \langle \bz_i, \bx\rangle} } \quad \text{ s.t. } A\bx=\boldsymbol{b},
\end{equation*}
\noindent where $\bz_i\in \mR^d$ is the sample point, $y_i\in\{-1,1\}$ is the label, and $A\in\mR^{m\times d}$ and~$\bb\in \mR^m$ form the deterministic constraints. We implement 8 datasets from LIBSVM \citep{Chang2011LIBSVM}: \texttt{austrilian}, \texttt{breast-cancer}, \texttt{diabetes}, \texttt{heart}, \texttt{ionosphere}, \texttt{sonar}, \texttt{splice}, and \texttt{svmguide3}. For each dataset, we set $m=5$ and generate random $A$ and $\boldsymbol{b}$ by~drawing each element from a standard normal distribution. We ensure that $A$ has full row~rank in all problems. For both algorithms and all problems, the initial iterate is set to be~all one vector of appropriate dimension. In each iteration, we select one sample at random to estimate the objective gradient (and Hessian if EstH or AveH is used).~A~budget of 20 epochs---the number of passes over the dataset---is used for both algorithms and all problems. We stop the iteration if $\|\nabla\mathcal{L}_k\|\leq 10^{-4}$ or the epoch budget is consumed.

We report the average of the KKT residuals over 5 runs in Figure~\ref{fig:KKTResidualLogit}. From the~figure, we observe that \alg\ with all four choices of $B_k$ consistently outperforms $\ell_1$-StoSQP when $\beta_k = 0.5$, $1.0$, and $k^{-0.6}$. When $\beta_k=k^{-0.8}$, \alg\ enjoys~a better performance by using the estimated Hessian or averaged Hessian. This~experiment further illustrates the promising~performance~of~our~method.

\begin{figure}[h]
\centering
\subfigure[Constant $\beta_k$]{\includegraphics[width=0.43\textwidth]{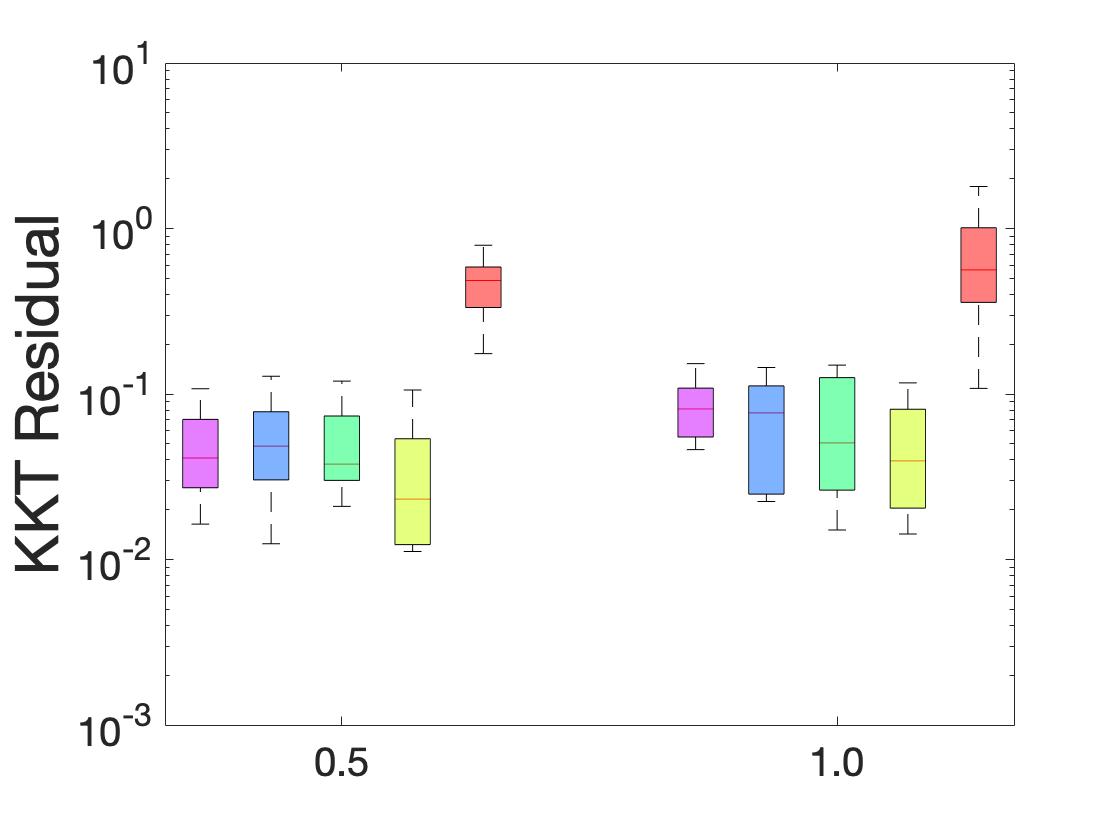}}
\subfigure[Decaying $\beta_k=k^{-s}$]{\includegraphics[width=0.43\textwidth]{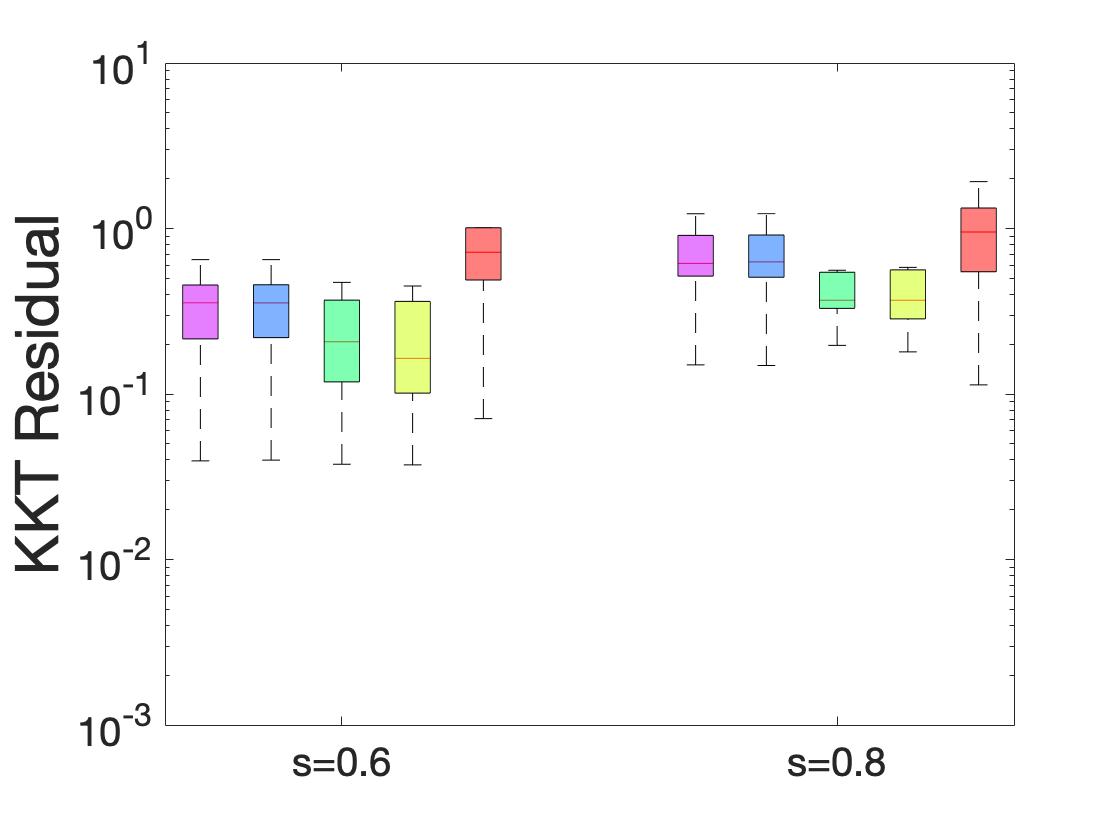}}
\subfigure{\includegraphics[width=0.7\textwidth]{Figures/labelLogit}}
\caption{KKT residual boxplots for constrained logistic regression problems. For each setup of $\beta_k$, there are five boxes. The first four boxes correspond to the proposed \alg\ method with four different choices of $B_k$, while the last box corresponds to the $\ell_1$-StoSQP method.}
\label{fig:KKTResidualLogit}

\end{figure}

%% file: sec6.tex
\section{Conclusion}\label{sec:6}

We designed a trust-region stochastic SQP (TR-StoSQP) algorithm to solve nonlinear optimization problems with stochastic objective and deterministic equality constraints. We developed an adaptive relaxation technique to~address the~infeasibility issue that arises when trust-region methods are applied to constrained problems. With a stabilized merit parameter, TR-StoSQP converges in two regimes. (i) When $\beta_k=\beta$, $\forall k\geq0$, the expectation of weighted averaged KKT residuals converges to a neighborhood around zero. (ii) When $\beta_k$ satisfies $\sum \beta_k = \infty$~and $\sum \beta_k^2<\infty$, the KKT residuals converge to zero almost surely. We also~showed~that~the merit parameter is ensured to stabilize, provided~the~\mbox{gradient} estimates are~bounded. Our numerical experiments on a subset of problems of the~CUTEst set and constrained logistic regression problems showed promising performance of the proposed method.

There are still several interesting future directions. 
First, it is of interest to~design trust-region StoSQP algorithms when the Jacobians of constraints are~\mbox{rank-deficient}. 
Second, how to establish global convergence without the assumption of~bounded~noise remains an open question. Removing that assumption may require a deeper understanding of the merit function and randomness in estimation.~Finally, it is of~\mbox{interest}~to devise a method that uses second-order information efficiently. To fully exploit second-order derivatives, the method should move the trial steps along the negative curvature appropriately.

%% file: appendix.tex
\section{Additional Analysis of the Behavior of the Merit \mbox{Parameter}}\label{appendix:A}

In the appendix, we further investigate the stability behavior of the merit parameter when using the alternative two approaches in Remark \ref{rem:3} to decompose the~radius. As mentioned, for both approaches, the global convergence analysis directly follows from Section \ref{subsec:global_conv}. 

We first show that for the method in Remark \ref{rem:3}(i), the merit parameter will~stabilize under Assumption \ref{ass:bound_error}.

\begin{lemma}\label{lemma:predk_2}
Suppose Assumptions \ref{ass:1-1} and \ref{ass:bound_error} hold and the relaxation technique in Remark \ref{rem:3}(i) is employed. Then, there exist a (potentially random) $\barK<\infty$ and~a deterministic constant $\hat{\mu}$, such that $\barmu_k=\barmu_{\barK}\leq \hatmu$, $\forall k>\barK$. 
\end{lemma}

\begin{proof}
Similar to Lemma \ref{lemma:predk}, we only show that there exists a \mbox{deterministic}~threshold $\tilde{\mu}>0$ independent of $k$ such that \eqref{eq:upper_bound_predk} is satisfied as long as $\barmu_k\geq \tilde{\mu}$. Using the same derivation as Lemma \ref{lemma:predk}, we have
\begin{align*}
\text{Pred}_k & \leq -\|\bar{\nabla}_{\bx}\L_k\|\tilde{\Delta}_k+\bargamma_k\|B_k\|\|\bv_k\|\tilde{\Delta}_k+\frac{1}{2}\|B_k\|\tilde{\Delta}_k^2 +\bargamma_k(M_1+\kappa_{\nabla f})\|\bv_k\| \\
&\quad + \frac{1}{2}\bargamma_k\|B_k\|\|\bv_k\|^2 -\barmu_k\bargamma_k\|c_k\|\\
& \leq -\|\bar{\nabla}_{\bx}\L_k\|\Delta_k+\bargamma_k\|\bv_k\|\|\bar{\nabla}_{\bx}\L_k\|+\bargamma_k\|B_k\|\|\bv_k\|\tilde{\Delta}_k+\frac{1}{2}\|B_k\|\tilde{\Delta}_k^2\\
& \quad  +\bargamma_k(M_1+\kappa_{\nabla f})\|\bv_k\| + \frac{1}{2}\bargamma_k\|B_k\|\|\bv_k\|^2 -\barmu_k\bargamma_k\|c_k\| \quad(\text{since }\tilde{\Delta}_k\geq \Delta_k-\bargamma_k\|\bv_k\|)\\
& = -\|\bar{\nabla}_{\bx}\L_k\|\Delta_k-\|c_k\|\Delta_k+\|c_k\|\Delta_k+\bargamma_k\|\bv_k\|\|\bar{\nabla}_{\bx}\L_k\|+\bargamma_k\|B_k\|\|\bv_k\|\tilde{\Delta}_k\\
& \quad  +\frac{1}{2}\|B_k\|\tilde{\Delta}_k^2+\bargamma_k(M_1+\kappa_{\nabla f})\|\bv_k\| + \frac{1}{2}\bargamma_k\|B_k\|\|\bv_k\|^2 -\barmu_k\bargamma_k\|c_k\|\\
& \leq -\|\bar{\nabla}\L_k\|\Delta_k+\frac{1}{2}\|B_k\|\Delta_k^2+\|c_k\|\Delta_k+\bargamma_k\|\bv_k\|\|\bar{\nabla}_{\bx}\L_k\|+\bargamma_k\|B_k\|\|\bv_k\|\Delta_k\\
& \quad  +\bargamma_k(M_1+\kappa_{\nabla f})\|\bv_k\| + \frac{1}{2}\bargamma_k\|B_k\|\|\bv_k\|^2 -\barmu_k\bargamma_k\|c_k\|,
\end{align*}
since $\|\bar{\nabla}_{\bx}\L_k\|+\|c_k\|\geq \|\bar{\nabla}\L_k\|$ and $\tilde{\Delta}_k\leq\Delta_k$. Thus, \eqref{eq:upper_bound_predk} holds as long as
\begin{equation*}
 \barmu_k\bargamma_k\|c_k\|\geq    \|c_k\|\Delta_k+\bargamma_k\|\bv_k\|\|\bar{\nabla}_{\bx}\L_k\|+\bargamma_k\|B_k\|\|\bv_k\|\Delta_k +\bargamma_k(M_1+\kappa_{\nabla f})\|\bv_k\| + \frac{\bargamma_k}{2}\|B_k\|\|\bv_k\|^2.
\end{equation*}
Since $\|\bv_k\|\leq \|c_k\|/\sqrt{\kappa_{1,G}}$, $\|\bar{\nabla}_{\bx}\L_k\|\leq \|\nabla_{\bx}\L_k\|+\|\nabla f_k-\barg_k\|\leq \kappa_{\nabla f}+M_1$ and~$ \Delta_k\leq \Delta_{\max}$, it is sufficient to show
\begin{equation*}
\barmu_k\bargamma_k\|c_k\|\geq    \|c_k\|\Delta_k+\bargamma_k\|c_k\|\left(\frac{\kappa_B\Delta_{\max}+2(M_1+\kappa_{\nabla f})}{\sqrt{\kappa_{1,G}}}+\frac{\kappa_B\kappa_c}{2\kappa_{1,G}}\right).
\end{equation*}
Equivalently, 
\begin{equation*}
\barmu_k\geq \frac{\Delta_k}{\bargamma_k}+\left(\frac{\kappa_B\Delta_{\max}+2(M_1+\kappa_{\nabla f})}{\sqrt{\kappa_{1,G}}}+\frac{\kappa_B\kappa_c}{2\kappa_{1,G}}\right).
\end{equation*}
Here, we only consider $\|c_k\|\neq 0$ since the result trivially holds when $\|c_k\|=0$. From \eqref{Ful_GenerateRadius}, we find that 
\begin{equation*}
\frac{\Delta_k}{\bargamma_k}\leq \frac{\eta_{1,k}\alpha_k\|\bar{\nabla}\L_k\|}{\bargamma_k}.
\end{equation*}
By \eqref{equ:4.13}, $\bargamma_k\geq\frac{1}{2}\zeta\phi_k\alpha_k=\frac{1}{2}\zeta\min\{\|B_k\|/\|G_k\|,1\}\alpha_k$. Noting that $\min\{\|B_k\|/\|G_k\|,1\}\geq \min\{1/(\kappa_B\sqrt{\kappa_{2,G}}),1\}$ and $\|\bar{\nabla}\L_k\|\leq \kappa_c+M_1+\kappa_{\nabla f}$, we obtain
\begin{equation*}
\frac{\Delta_k}{\bargamma_k}\leq \frac{2\eta_{\max}}{\zeta}(\kappa_c+\kappa_{\nabla f}+M_1)\cdot \max\{\kappa_B\sqrt{\kappa_{2,G}},1\}.
\end{equation*}
Therefore, \eqref{eq:upper_bound_predk} holds as long as
\begin{multline*}
\barmu_k\geq \tilde{\mu}\coloneqq\frac{2\eta_{\max}}{\zeta}(\kappa_c+\kappa_{\nabla f}+M_1)\cdot \max\{\kappa_B\sqrt{\kappa_{2,G}},1\}
+\left(\frac{\kappa_B\Delta_{\max}+2(M_1+\kappa_{\nabla f})}{\sqrt{\kappa_{1,G}}}+\frac{\kappa_B\kappa_c}{2\kappa_{1,G}}\right).
\end{multline*}
Since $\barmu_k$ is increased by at least a factor of $\rho$ for each update, we define $\hat{\mu}\coloneqq\rho\tilde{\mu}$ and complete the proof.
\end{proof}

We then show that for the method in Remark \ref{rem:3}(ii), the merit parameter will~stabilize just under Assumption \ref{ass:bound_error}(i). However, a tuning parameter $\theta\in(0,1)$ is involved to control the length of the normal step.

\begin{lemma}\label{lemma:predk_3}
Suppose Assumptions \ref{ass:1-1} and \ref{ass:bound_error}(i) hold and the relaxation technique in Remark \ref{rem:3}(ii) is employed. Then, there exist a (potentially random) $\barK<\infty$ and a deterministic constant $\hat{\mu}$, such that $\barmu_k=\barmu_{\barK}\leq \hatmu$, $\forall k>\barK$. 
\end{lemma}

\begin{proof}
Similar to Lemma \ref{lemma:predk}, we only show that there exists a \mbox{deterministic}~threshold $\tilde{\mu}>0$ independent of $k$ such that \eqref{eq:upper_bound_predk} is satisfied as long as $\barmu_k\geq \tilde{\mu}$. Using the same derivation as Lemma \ref{lemma:predk_2}, we only need to show
\begin{equation*}
\barmu_k\geq \frac{\Delta_k}{\bargamma_k}+\left(\frac{\kappa_B\Delta_{\max}+2(M_1+\kappa_{\nabla f})}{\sqrt{\kappa_{1,G}}}+\frac{\kappa_B\kappa_c}{2\kappa_{1,G}}\right)
\end{equation*}
holds for $\barmu_k$ larger than a deterministic threshold for $\|c_k\|\neq 0$. Since for $\forall k\geq 0$, 
\begin{equation*}
\frac{\Delta_k}{\bargamma_k}\leq \frac{\eta_{1,k}\alpha_k\|\bar{\nabla}\L_k\|}{\bargamma_k}.
\end{equation*}
By the projection technique of choosing $\bargamma_k$ and the fact that $\eta_{2,k}\geq \eta_{1,k}/2$, we have
\begin{equation*}
\frac{\breve{\Delta}_k}{\|\bv_k\|} = \frac{\theta\Delta_k}{\|\bv_k\|}\stackrel{\eqref{Ful_GenerateRadius}}{\geq}\frac{\theta\eta_{2,k}\alpha_k\|\bnabla\mL_k\|}{\|\bv_k\|}\geq \frac{\theta\eta_{1,k}\alpha_k\|c_k\|}{2\|\bv_k\|} \stackrel{\eqref{def:eta2k}}{=}\frac{\theta\zeta\alpha_k}{2}.
\end{equation*}
Further, since $\theta\zeta\alpha_k/2\leq 1$, we know $\theta\zeta\alpha_k/2\leq \bargamma_k^{\text{trial}}$, implying $\bargamma_k\geq\theta\zeta\alpha_k/2$. Thus,
\begin{equation*}
\frac{\Delta_k}{\bargamma_k}\leq \frac{2\eta_{\max}}{\zeta\theta}(\kappa_c+\kappa_{\nabla f}+M_1).
\end{equation*}
Therefore, \eqref{eq:upper_bound_predk} holds as long as
\begin{equation*}
\barmu_k\geq \tilde{\mu}\coloneqq\frac{2\eta_{\max}}{\zeta\theta}(\kappa_c+\kappa_{\nabla f}+M_1)
+\left(\frac{\kappa_B\Delta_{\max}+2(M_1+\kappa_{\nabla f})}{\sqrt{\kappa_{1,G}}}+\frac{\kappa_B\kappa_c}{2\kappa_{1,G}}\right).
\end{equation*}
Since $\barmu_k$ is increased by at least a factor of $\rho$ for each update, we define $\hat{\mu}\coloneqq\rho\tilde{\mu}$ and complete the proof.
\end{proof}